\documentclass[11pt]{amsart}
\usepackage[T1]{fontenc}
\usepackage{lmodern}
\usepackage[utf8]{inputenc}
\usepackage{amsmath,amssymb,amsthm,mathtools}
\usepackage[margin=1.1in]{geometry}
\usepackage{microtype}
\usepackage[colorlinks=true,linkcolor=black,citecolor=black,urlcolor=black]{hyperref}
\numberwithin{equation}{section}
\newcommand{\R}{\mathbb R}
\newcommand{\Z}{\mathbb Z}

\newcommand{\Q}{\mathbb Q}
\newcommand{\T}{\mathbb T}
\newcommand{\Id}{\mathrm{Id}}
\newcommand{\Lip}{\mathrm{Lip}}
\newcommand{\eps}{\varepsilon}

\newtheorem{Theorem}{Theorem}
\newtheorem{Proposition}{Proposition}[section]
\newtheorem{Lemma}[Proposition]{Lemma}

\theoremstyle{definition}
\newtheorem{Definition}[Proposition]{Definition}
\newtheorem{Remark}[Proposition]{Remark}
\newtheorem{Question}{Question}
\DeclareMathOperator{\Leb}{Leb}
\DeclareMathOperator{\dist}{dist}

\begin{document}
\title[Invariant graphs for dissipative twist maps]{On the dynamics of invariant graphs for dissipative twist maps}
\author{Qi Li}
\address{School of Mathematics and Statistics, Beijing Institute of Technology, Beijing 100081, China}
\email{qilicindy@bit.edu.cn}
\author{Lin Wang$^*$}
\address{School of Mathematics and Statistics, Beijing Institute of Technology, Beijing 100081, China}
\email{lwang@bit.edu.cn}
\subjclass[2020]{Primary 37E40; Secondary 37D10, 37J40}
\keywords{Dissipative twist map, invariant graph, normal hyperbolicity}
\thanks{$^*$Corresponding author: Lin Wang, lwang@bit.edu.cn}

\begin{abstract}
We study the persistence and internal dynamics of invariant graphs for dissipative twist maps. Explicit $C^1$ persistence and parameter estimates allow every real rotation number to be realized by varying either translation parameter, while arbitrarily small zero-mean analytic perturbations at fixed irrational horizontal parameters can yield analytic graphs with dynamics conjugate to rational rotations. With a fixed Brjuno horizontal parameter and fixed-strip analyticity, normally attracting $C^r$ graphs persist under perturbations of size $(1-\lambda)^\gamma$, for any $\gamma>0$ and prescribed finite $r$, as $\lambda\to1^-$. At the critical Lipschitz bound, the invariant graphs can carry Denjoy dynamics and have countably many nondifferentiable points, with perturbations smooth outside a null Cantor set and tending to zero in every H\"older norm of order less than one.
\end{abstract}
\maketitle

\begingroup
\fontsize{9.5}{11}\selectfont
\tableofcontents
\endgroup
\clearpage
\section{Introduction}
Invariant graphs in dissipative twist maps connect the geometry of
attracting sets with the dynamics of circle homeomorphisms. Their
persistence is governed by transverse contraction, whereas the dynamics
on a persisting graph may change under perturbation. These two questions
become particularly closely related as the dissipation tends to zero or
the regularity of the perturbation decreases. The study of invariant
curves and Birkhoff attractors provides a geometric background
\cite{B,L1,L2}; normal hyperbolicity and its persistence theory provide
local existence and regularity results \cite{HPS,BB13}, with a converse
characterization due to Ma\~n\'e \cite{Man}.
Conformally symplectic KAM theory \cite{CCD1,CCD22}, normal-form and
translated-curve methods \cite{Mas18,Ma}, and variational methods
\cite{MS} give further approaches to invariant sets in this setting.

The existence of a smooth invariant graph does not determine its rotation
number. In particular, an irrational frequency of the unperturbed map
need not remain the frequency of the perturbed dynamics. This distinction
is relevant to both periodic orbits and the arithmetic assumptions used
in circle linearization \cite{Ar,H11,Y1}. At lower regularity, wandering
intervals and exceptional minimal sets are also possible \cite{De,de}.
Regularity affects persistence itself: Sorrentino and Wang \cite{SW26}
construct perturbations that destroy all continuous invariant graphs and
are arbitrarily small in any prescribed $C^\beta$ topology,
$0<\beta<1$. In the $C^r$ category, Kong, Massetti and Wang~\cite{KMW26}
established quantitative conditions for the persistence of
$C^r$ invariant graphs of dissipative twist maps on
$\mathbb{T}^d\times\mathbb{R}^d$, for $r\geq 1$.

We consider the family
\begin{equation}\label{Fp}
 F^\phi_{\lambda,\alpha}(x,y)
 =\bigl(x+\alpha_1+\lambda y+\phi(x),\,
         \alpha_2+\lambda y+\phi(x)\bigr),
\end{equation}
where $\T:=\R/\Z$, $0<\lambda<1$,
$\alpha=(\alpha_1,\alpha_2)\in\R^2$, and $\phi$ is a $1$-periodic
function satisfying
\begin{equation}\label{zero-mean}
 \int_0^1\phi(x)\,dx=0.
\end{equation}
We use the same notation for the lift \eqref{Fp} and the induced map on
$\T\times\R$, and write $h=1-\lambda$. If $\phi\in C^1(\T)$, then
\[
 \det DF^\phi_{\lambda,\alpha}=\lambda,
 \qquad \partial_y(F^\phi_{\lambda,\alpha})_1=\lambda>0,
\]
so \eqref{Fp} is a positive dissipative twist map. It has an explicit
inverse: for $(X,Y)=F^\phi_{\lambda,\alpha}(x,y)$,
\[
 x=X-Y-\alpha_1+\alpha_2,
 \qquad y=\lambda^{-1}\bigl(Y-\alpha_2-\phi(x)\bigr).
\]
In particular, it remains a global homeomorphism when $\phi$ is merely
continuous.

\begin{Definition}
A continuous invariant graph is a set
\[
 \Gamma_\Psi=\{(x,\Psi(x)):x\in\T\},\qquad \Psi\in C^0(\T),
\]
such that $F^\phi_{\lambda,\alpha}(\Gamma_\Psi)=\Gamma_\Psi$ on the cylinder.
Its induced lift is
\begin{equation}\label{induced-lift}
 g(x)=x+\alpha_1+\lambda\Psi(x)+\phi(x).
\end{equation}
The graph invariance equation is
\begin{equation}\label{invariance}
 \Psi(g(x))=\alpha_2+\lambda\Psi(x)+\phi(x).
\end{equation}
The induced circle map is an orientation-preserving homeomorphism.
For its lift, we use the real rotation number
\[
 \rho(g)=\lim_{n\to\infty}\frac{g^n(x)-x}{n},
\]
which is independent of $x$; the circle rotation number is its class modulo $\Z$.
\end{Definition}

The unperturbed map
\begin{equation}\label{F1}
 T_{\lambda,\alpha}(x,y)
 =(x+\alpha_1+\lambda y,\alpha_2+\lambda y)
\end{equation}
has the unique invariant graph
\begin{equation}\label{ingm1}
 \Gamma_0=\T\times\left\{\frac{\alpha_2}{1-\lambda}\right\}.
\end{equation}
Its dynamics is the rigid rotation of frequency
$\alpha_1+\lambda\alpha_2/(1-\lambda)$.
Its tangent and transverse multipliers are $1$ and $\lambda$,
respectively. Our results concern quantitative persistence, dependence on
$\alpha$, changes of the internal rotation number at fixed $\alpha$, and
the regularity of graphs carrying Denjoy dynamics.

\subsection{Notation}
We identify circle functions with their $1$-periodic lifts. An increasing
degree-one lift satisfies $g(x+1)=g(x)+1$. We write
$R_\omega(x):=x+\omega$ for the rigid rotation lift; the same notation is
used for its circle map. Set
\[
 \|t\|_{\R/\Z}:=\min_{j\in\Z}|t-j|,\qquad
 \dist(x,y):=\|x-y\|_{\R/\Z},
\]
\[
 \langle u\rangle:=\int_0^1u(x)\,dx,\qquad
 \widehat u_k:=\int_0^1u(x)e^{-2\pi ikx}\,dx.
\]
For a nonnegative integer $k$, we use
$\|u\|_{C^k}=\max_{0\le j\le k}\|u^{(j)}\|_\infty$.
For maps of two variables, the norm is the maximum of the suprema of all
componentwise partial derivatives of total order at most $k$.
The notation
\[
 \|u\|_{\Lip}=\sup_{x\ne y}\frac{|u(x)-u(y)|}{|x-y|}
\]
denotes the Lipschitz seminorm of the periodic lift. For $0<\beta<1$,
\[
 [u]_{C^\beta}
 =\sup_{x\ne y\in\T}\frac{|u(x)-u(y)|}{\dist(x,y)^\beta},
 \qquad \|u\|_{C^\beta}=\|u\|_\infty+[u]_{C^\beta}.
\]
For $s>0$, let
\begin{equation}\label{analytic-norm}
 \T_s=\{z\in\mathbb C/\Z:|\operatorname{Im}z|<s\},
 \qquad \|u\|_s=\sup_{z\in\T_s}|u(z)|.
\end{equation}
An analytic perturbation on $\T_s$ is assumed bounded and holomorphic
there and real-valued on $\T$. In statements involving a fixed strip,
$s$ is independent of the perturbation and of $\lambda$.

\subsection{Persistence and parameter dependence}

The rotation number of the unperturbed graph depends on both
translation parameters. We first ask whether this freedom persists
under perturbation.
\begin{Question}\label{question:prescription}
{\it For a fixed small perturbation, can every real rotation number be realized
on the persisting invariant graph by varying either translation parameter
while keeping the other fixed?}
\end{Question}
The following quantitative statement answers this question affirmatively.

\begin{Theorem}\label{M1}
Fix $0<\lambda<1$ and choose
\begin{equation}\label{epsilon-range}
 0<\eps<\min\{1,\lambda^{-1/2}-1\},\qquad
 \delta=\frac{\eps(1-\lambda-\lambda\eps)}{1+\eps}.
\end{equation}
Suppose that $\phi\in C^1(\T)$ satisfies \eqref{zero-mean} and
$\|\phi\|_{C^1}\le\delta$. For each $\alpha\in\R^2$, let
$\Gamma_{\Psi_\alpha}$ be the unique continuous invariant graph
provided by Proposition~\ref{prop:preparatory-persistence}, and let
$g_\alpha$ be its induced lift, defined by \eqref{induced-lift}.
Then the following assertions hold.
\begin{enumerate}
\renewcommand{\labelenumi}{(\arabic{enumi})}
\item\label{item:M1-parameters} For all $\alpha,\alpha'\in\R^2$,
\begin{equation}\label{parameter-bound}
 \|\Psi_\alpha-\Psi_{\alpha'}\|_\infty
 \le \frac{\eps|\alpha_1-\alpha_1'|+|\alpha_2-\alpha_2'|}
 {1-\lambda(1+\eps)}.
\end{equation}
\item\label{item:M1-rotations} If either $\alpha_1$ or $\alpha_2$ is fixed, varying the other parameter
realizes every real rotation number for the induced lift.
\end{enumerate}
\end{Theorem}

The existence, uniqueness and $C^1$ bound used here are established
separately in Proposition~\ref{prop:preparatory-persistence}.
Theorem~\ref{M1} records the resulting parameter dependence and
surjectivity of the rotation number.

The restriction in \eqref{epsilon-range} gives a height contraction
factor $\lambda(1+\eps)$ and, over the limiting base map, a
slope-operator contraction factor $\lambda(1+\eps)^2$. The function
$\eps\mapsto\eps(1-\lambda-\lambda\eps)/(1+\eps)$ reaches
$(1-\sqrt\lambda)^2$ at $\eps=\lambda^{-1/2}-1$.
For fixed $\lambda$, its expansion is
$\delta=(1-\lambda)\eps+O(\eps^2)$ as $\eps\to0$.
\subsection{Weak dissipation and analytic perturbations}

As $\lambda\to1^-$, the transverse contraction of the unperturbed
graph weakens and the persistence bound in Proposition~\ref{prop:preparatory-persistence} tends to zero.
\begin{Question}\label{question:weak-dissipation}
{\it How does the persistence threshold depend on $1-\lambda$? Under
fixed-strip analyticity and a Brjuno condition on the external
parameter, can it exceed the scale supplied by a uniform Lipschitz estimate?}
\end{Question}
We first recall the uniform Lipschitz persistence bound associated with the graph
transform in \cite{HLW}.

\begin{Proposition}\label{M2}
Let $0<\lambda<1$ and $\alpha\in\R^2$. If $\phi$ is $1$-periodic and
Lipschitz, satisfies \eqref{zero-mean}, and
\begin{equation}\label{critical-bound}
 \|\phi\|_{\Lip}\le(1-\sqrt\lambda)^2,
\end{equation}
then $F^\phi_{\lambda,\alpha}$ has a unique continuous invariant graph.
It is Lipschitz and satisfies
\[
 \|\Psi\|_{\Lip}\le\lambda^{-1/2}-1,
 \qquad
 \left\|\Psi-\frac{\alpha_2}{1-\lambda}\right\|_\infty
 \le\frac{\|\phi\|_\infty}{1-\lambda}.
\]
If $\phi\in C^1(\T)$ and the inequality in
\eqref{critical-bound} is strict, then the graph is
$C^1$ and normally attracting.
\end{Proposition}

The Lipschitz existence assertion uses only the estimates in the graph
transform. We call $(1-\sqrt\lambda)^2$ the
\emph{critical Lipschitz bound}, referring to the endpoint of this
uniform estimate. It is asymptotic to $(1-\lambda)^2/4$ as
$\lambda\to1^-$. With analyticity on a fixed strip and an arithmetic
condition on the external frequency, the next theorem allows a larger
perturbation scale.

An irrational number $\alpha_1$ is called a Brjuno number if the
denominators $q_n$ of its continued-fraction convergents satisfy
\begin{equation}\label{eq:brjuno-condition}
 \sum_{n\ge0}\frac{\log q_{n+1}}{q_n}<\infty.
\end{equation}

\begin{Theorem}
\label{M222}
Let $\alpha_1$ be Brjuno, $s>0$, and $\gamma>0$.
For every integer $r\ge1$, there are $\lambda_0\in[1/2,1)$ and $C>0$
such that, for every $\lambda\in[\lambda_0,1)$ and every $1$-periodic,
zero-mean perturbation $\phi^\lambda$ that is bounded and holomorphic on
$\T_s$, real-valued on $\T$, and satisfies
\[
 \|\phi^\lambda\|_s\le(1-\lambda)^\gamma,
\]
the map $F^{\phi^\lambda}_{\lambda,(\alpha_1,0)}$ has a unique
continuous invariant graph $\Gamma_{\Psi_\lambda}$.
This graph is normally attracting, belongs to $C^r$, and satisfies
\[
 \|\Psi_\lambda\|_{C^0}\le C\|\phi^\lambda\|_s.
\]
The real rotation number $\rho_\lambda$ of the induced circle
diffeomorphism satisfies
\[
 |\rho_\lambda-\alpha_1|\le C\|\phi^\lambda\|_s^2.
\]
The constants depend only on $r,\gamma,s$ and the fixed number
$\alpha_1$.
\end{Theorem}

For $0<\gamma<2$, the admissible bound is larger than the scale in
Proposition~\ref{M2}. The proof combines finite-order averaging with a
quantitative graph transform. The averaging order is chosen so that the
remaining angle dependence is small relative to the dissipation gap.
The zero-mean condition then yields an exact mean-displacement identity,
which gives the quadratic rotation-number estimate.

Massetti \cite{Ma} combines normal forms for translated curves with
graph transforms to obtain persistence regions for more general
perturbations of dissipative twist maps. For \eqref{Fp} with
$\alpha_2=0$, Theorem~\ref{M222} uses the exact conformally symplectic
structure while keeping both translation parameters fixed. It assumes
a Brjuno external frequency and allows every positive exponent $\gamma$
for each prescribed finite regularity.

\begin{Remark}
The Brjuno hypothesis concerns the fixed external frequency
$\alpha_1$. No arithmetic condition on the actual internal rotation
number $\rho_\lambda$ is imposed, and that rotation number may be
rational. The proof uses finitely many cohomological inversions,
each with a positive strip loss.
\end{Remark}

\begin{Remark}
\label{rem:comparison-kam}
The conformally symplectic KAM theorems of Calleja, Celletti and
de la Llave \cite{CCD1,CCD22} prescribe a Diophantine internal frequency
and solve the invariance equation in a family with an adjustable drift
parameter. Under their a posteriori accuracy and nondegeneracy
hypotheses, the analytic invariant torus carries dynamics analytically
conjugate to the prescribed rotation. These hypotheses do not require
proximity to an integrable system.

In Theorem~\ref{M222}, both translation parameters remain fixed at
$(\alpha_1,0)$, and the internal rotation number is allowed to change.
The subexponential small-divisor estimate implied by the Brjuno condition
suffices for the finite averaging construction; normal contraction then
controls the remainder. Thus the result concerns persistence without
frequency adjustment, and it does not assert an analytic conjugacy.
Whenever a KAM invariant circle is a graph for the same map, uniqueness
identifies it with the graph of Theorem~\ref{M222}.
\end{Remark}

\subsection{Frequency change at fixed parameters}
The distinction between external and internal frequencies is already
visible for arbitrarily small analytic perturbations. The next theorem
holds at every irrational external parameter, with no arithmetic
restriction.
\begin{Question}\label{question:frequency-change}
{\it For an irrational horizontal parameter, can arbitrarily small
analytic perturbations change the rotation number of the invariant
graph?}
\end{Question}
The construction starts with a circle map conjugate to a rotation and
recovers the graph and perturbation from the invariance equation. An
exact mean-displacement condition enforces \eqref{zero-mean} while
keeping $\alpha_1$ fixed. Proposition~\ref{prop:explicit-frequency-shift}
gives the resulting frequency change explicitly. Since the perturbations
are small on a fixed complex strip, the same sequence tends to zero in
every finite $C^k$ norm. A further perturbation yields attracting and
repelling periodic orbits on the graph, corresponding respectively to
sinks and saddles of the cylinder map.

\begin{Theorem}\label{Mt3}
Fix $0<\lambda<1$, an irrational $\alpha_1$, and $s>0$.
There exist zero-mean real-analytic perturbations $\phi_n$, bounded and
holomorphic on $\T_s$, such that
\[
 \|\phi_n\|_s\longrightarrow0,
\]
and each $F^{\phi_n}_{\lambda,(\alpha_1,0)}$ has a unique continuous
invariant graph. This graph is real analytic, and its induced circle
map is analytically conjugate to a rational rotation. In particular,
its real rotation number differs from $\alpha_1$.
\end{Theorem}

The theorem answers Question~\ref{question:frequency-change}
affirmatively for every irrational $\alpha_1$.

\subsection{Lipschitz graphs with Denjoy dynamics}
At lower regularity, the internal dynamics may have wandering intervals.
A \emph{monotone semiconjugacy} from an orientation-preserving circle
homeomorphism $g$ to $R_\omega$ is a continuous nondecreasing degree-one
surjection $P$ satisfying $P\circ g=R_\omega\circ P$.
An open interval is \emph{wandering} if its iterates are pairwise disjoint.
We call $g$ a \emph{Denjoy homeomorphism} if its rotation number is
irrational and it has a wandering interval.

\begin{Theorem}\label{M3}
Fix $\lambda\in(0,1)$, $\omega\in\R\setminus\Q$, and $\beta\in(0,1)$,
and put $A=(1-\sqrt\lambda)^2$.
For all sufficiently large integers $N$, there exist a zero-mean
Lipschitz function $\phi_N$, a null Cantor set $K_N\subset\T$, and
$\alpha_N\in\R$ such that $F^{\phi_N}_{\lambda,(\alpha_N,0)}$ has a
Lipschitz invariant graph $\Gamma_{\Psi_N}$ with the following properties.
\begin{enumerate}
\item The perturbation satisfies
\[
 \|\phi_N\|_{\Lip}=A,
 \qquad
 \|\phi_N\|_{C^\beta}
 =O_{\lambda,\beta}(N^{-(1-\beta)}),
 \qquad
 \alpha_N=\omega+O_\lambda(N^{-1}).
\]
\item Both $\phi_N$ and $\Psi_N$ are affine on every component of
$\T\setminus K_N$, and hence are $C^\infty$ there.
The nondifferentiability set of $\Psi_N$ consists exactly of the
endpoints of these components. It is the union of two countably infinite
orbits and is dense in $K_N$.
\item The induced circle map is a bi-Lipschitz Denjoy homeomorphism
with rotation number $\omega$. It is monotonically semiconjugate to
$R_\omega$ and has wandering intervals, so it is not topologically
conjugate to $R_\omega$.
\end{enumerate}
\end{Theorem}

The set $K_N$ is the minimal set of the induced map. The graph is
differentiable at every point of $K_N$ other than the gap endpoints.
The perturbations in Theorem~\ref{M3} are Lipschitz but are not $C^1$;
the result therefore concerns a different regularity class from the
 $C^1$ persistence assertion in Proposition~\ref{M2}.
The construction attains the critical seminorm exactly while the full
$C^\beta$ norm tends to zero for each $\beta<1$.

\subsection*{Organization}
Section~\ref{S3} establishes Proposition~\ref{prop:preparatory-persistence}
by a graph transform, and then proves the two assertions of
Theorem~\ref{M1}.
Section~\ref{normal-section} proves Theorem~\ref{M222} by finite-order
averaging and a quantitative graph transform.
Section~\ref{circle-section} proves Theorem~\ref{Mt3} using exact frequency
identities and the construction in Proposition~\ref{prop:explicit-frequency-shift}.
Section~\ref{S5} proves Theorem~\ref{M3} by constructing a Denjoy map
and recovering the perturbation from the graph equation.
\section{Persistence and parameter dependence}\label{S3}

We first establish the graph equation and uniqueness for continuous
periodic perturbations. We then prove the quantitative persistence
statement in Proposition~\ref{prop:preparatory-persistence}, which
supplies the height and slope bounds needed for the two assertions
of Theorem~\ref{M1}.
\subsection{The graph equation and uniqueness}\label{sec:graph-uniqueness}

Write $F:=F^\phi_{\lambda,\alpha}$ and suppose that $\Gamma_\Psi$ is a
continuous invariant graph. The restriction of $F$ to $\Gamma_\Psi$ is a homeomorphism.
Its induced lift $g$ is continuous and injective; the identity
$g(x+1)=g(x)+1$ implies that it is increasing and onto.
Subtracting the two components of \eqref{Fp} and setting $X=g(x)$ gives
\begin{equation}\label{graph-from-g}
 \Psi(X)=X-g^{-1}(X)-\alpha_1+\alpha_2.
\end{equation}
Substitution in \eqref{induced-lift} yields
\begin{equation}\label{eq:main}
 g(x)+\lambda g^{-1}(x)
 =(1+\lambda)x+(1-\lambda)\alpha_1+\lambda\alpha_2+\phi(x).
\end{equation}
Conversely, if an increasing degree-one lift $g$ satisfies
\eqref{eq:main}, then \eqref{graph-from-g} defines a continuous periodic
function $\Psi$. These two identities imply \eqref{induced-lift} and
\eqref{invariance}, so $\Gamma_\Psi$ is invariant.

Uniqueness follows from contraction of area. The factorization
$F=S\circ K$, where
\[
 K(x,y):=(x,\alpha_2+\lambda y+\phi(x)),\qquad
 S(x,Y):=(x+\alpha_1-\alpha_2+Y,Y),
\]
gives, for every Borel set $E$ in the cylinder,
\[
 \operatorname{area}(K(E))=\lambda\operatorname{area}(E),
 \qquad
 \operatorname{area}(S(E))=\operatorname{area}(E).
\]
Indeed, these identities follow from Fubini's theorem and
one-dimensional changes of variables, integrating along vertical
fibers for $K$ and along horizontal circles for $S$.
Thus
$\operatorname{area}(F(E))=\lambda\operatorname{area}(E)$,
without any differentiability assumption on $\phi$.

If $\Gamma_\Psi$ is invariant, then its upper component
$U_\Psi:=\{(x,y):y>\Psi(x)\}$ is also invariant. To see this, observe
that $F$ permutes the two components of the complement of
$\Gamma_\Psi$, while its second component tends uniformly to $+\infty$
as $y\to+\infty$. It must therefore preserve the upper component.
For two invariant graphs $\Gamma_{\Psi_1}$ and $\Gamma_{\Psi_2}$, the
symmetric difference $U_{\Psi_1}\mathbin{\triangle}U_{\Psi_2}$ is
invariant and has finite area
\[
 \int_0^1|\Psi_1(x)-\Psi_2(x)|\,dx.
\]
Since area is multiplied by $\lambda<1$, this integral vanishes.
Continuity gives $\Psi_1=\Psi_2$.
We shall also use the following mean identities.
Let $g$ be an increasing bi-Lipschitz degree-one lift and
set $d=g-\Id$. Periodicity and the substitution $x=g(t)$ give
\begin{equation}\label{eq:inverse-displacement-average}
\begin{aligned}
 \int_0^1\bigl(g^{-1}(x)-x\bigr)\,dx
 &=\int_{g(0)}^{g(0)+1}\bigl(g^{-1}(x)-x\bigr)\,dx\\
 &=-\int_0^1d(t)\bigl(1+d'(t)\bigr)\,dt\\
 &=-\int_0^1d(t)\,dt.
\end{aligned}
\end{equation}
Indeed, $d$ is absolutely continuous and periodic, so
$\int_0^1dd'=0$.

If $\Psi$ and $\phi$ are related to $g$ by
\eqref{graph-from-g} and \eqref{eq:main}, integration yields
\begin{equation}\label{eq:mean-normalization}
\begin{aligned}
 \langle\phi\rangle
 &=(1-\lambda)\bigl(\langle g-\Id\rangle-\alpha_1\bigr)
   -\lambda\alpha_2,\\
 (1-\lambda)\langle\Psi\rangle
 &=\alpha_2+\langle\phi\rangle.
\end{aligned}
\end{equation}
These identities do not require a zero-mean assumption on $\phi$.
In particular, when $\alpha_2=0$,
\[
 \langle\phi\rangle=0
 \quad\Longleftrightarrow\quad
 \langle g-\Id\rangle=\alpha_1,
\]
and either condition implies $\langle\Psi\rangle=0$.

\subsection{A contraction on Lipschitz graphs}

\begin{Proposition}\label{prop:preparatory-persistence}
Fix $0<\lambda<1$, choose
\[
 0<\eps<\min\{1,\lambda^{-1/2}-1\},\qquad
 \delta=\frac{\eps(1-\lambda-\lambda\eps)}{1+\eps},
\]
and let $\phi\in C^1(\T)$ be $1$-periodic, with
$\int_0^1\phi(x)\,dx=0$ and $\|\phi\|_{C^1}\le\delta$.
For every $\alpha\in\R^2$, the map $F^\phi_{\lambda,\alpha}$ has a
unique continuous invariant graph $\Gamma_{\Psi_\alpha}$.
The graph belongs to $C^1$ and satisfies
\[
 \left\|\Psi_\alpha-\frac{\alpha_2}{1-\lambda}\right\|_{C^1}
 \le\eps.
\]
\end{Proposition}

\begin{proof}[Proof of Proposition~\ref{prop:preparatory-persistence}]
\[
 c:=\frac{\alpha_2}{1-\lambda},\qquad
 a:=\alpha_1+\lambda c.
\]
The translation $y=c+u$ transforms the map into
\[
 (x,u)\longmapsto
 \bigl(x+a+\lambda u+\phi(x),\lambda u+\phi(x)\bigr).
\]
It suffices to construct an invariant graph $u=U(x)$ for this map.
Define
\begin{equation}\label{eq:strict-contraction-constants}
 b:=\lambda\eps+\delta=\frac{\eps}{1+\eps},\qquad
 m:=1-b=\frac1{1+\eps},\qquad
 \kappa:=\frac{\lambda}{m},\qquad
 q:=\frac{\lambda}{m^2}.
\end{equation}
By the assumed bound on $\eps$, both $\kappa$ and $q$ belong to $(0,1)$.
Let $\mathcal X$ be the space of continuous periodic functions $U$
satisfying
$\|U\|_\infty\le\eps$ and $\|U\|_{\Lip}\le\eps$.
With the uniform distance, it is a closed subset of $C^0(\T)$ and is
therefore complete.

For $U\in\mathcal X$, put
\[
 G_U(x):=x+a+\lambda U(x)+\phi(x).
\]
For $x<z$,
\[
 m(z-x)\le G_U(z)-G_U(x)\le(1+b)(z-x).
\]
Hence $G_U$ is an increasing degree-one homeomorphism and $G_U^{-1}$
is $m^{-1}$-Lipschitz. The image of the graph of $U$ is consequently
the graph of
\begin{equation}\label{eq:elementary-graph-transform}
 (\mathcal TU)(X)
 :=\lambda U(G_U^{-1}(X))+\phi(G_U^{-1}(X))
 =X-a-G_U^{-1}(X).
\end{equation}
The first expression gives $\|\mathcal TU\|_\infty\le b\le\eps$.
Moreover, for $X=G_U(x)$ and $Z=G_U(z)$,
\[
 |(\mathcal TU)(Z)-(\mathcal TU)(X)|
 \le b|z-x|
 \le\frac bm|Z-X|
 =\eps|Z-X|.
\]
Thus $\mathcal T$ maps $\mathcal X$ into itself.

For $U,V\in\mathcal X$, the lower Lipschitz bound for $G_U$ implies
\[
 \|G_U^{-1}-G_V^{-1}\|_\infty
 \le m^{-1}\|G_U-G_V\|_\infty
 \le\kappa\|U-V\|_\infty.
\]
The second expression in \eqref{eq:elementary-graph-transform} then
gives
\begin{equation}\label{eq:graph-transform-contraction}
 \|\mathcal TU-\mathcal TV\|_\infty
 \le\kappa\|U-V\|_\infty.
\end{equation}
Starting with $U_0:=0$, define $U_{n+1}:=\mathcal TU_n$. The estimate
\[
 \|U_{n+1}-U_n\|_\infty
 \le\kappa^n\|U_1-U_0\|_\infty
\]
and convergence of the geometric series show that $(U_n)$ converges
uniformly to a function $U\in\mathcal X$.
By \eqref{eq:graph-transform-contraction}, $\mathcal TU=U$ and
\begin{equation}\label{eq:graph-iterate-rate}
 \|U_n-U\|_\infty
 \le\kappa^n\|U_0-U\|_\infty
 \le\eps\kappa^n.
\end{equation}

It remains to prove $U\in C^1$. Each iterate $U_n$ is $C^1$:
this follows inductively from $U_0=0$, the bound $G_{U_n}'\ge m$,
and the inverse derivative formula
\[
 (G_{U_n}^{-1})'
 =\frac{1}{G_{U_n}'\circ G_{U_n}^{-1}}.
\]
Set
\[
 G:=G_U,\qquad G_n:=G_{U_n},\qquad
 H(p,v):=\frac{\lambda p+v}{1+\lambda p+v}.
\]
For $|p|\le\eps$ and $|v|\le\delta$, the denominator is at least $m$,
$|H(p,v)|\le b/m=\eps$, and
\begin{equation}\label{eq:slope-derivative-bounds}
 |\partial_pH(p,v)|\le q,\qquad
 |\partial_vH(p,v)|\le m^{-2}.
\end{equation}
Define, on the closed ball
$\{p\in C^0(\T):\|p\|_\infty\le\eps\}$,
\[
 (\mathcal Sp)(X)
 :=H\bigl(p(G^{-1}(X)),\phi'(G^{-1}(X))\bigr).
\]
This operator preserves the ball and contracts the uniform distance
by $q$. The same iteration argument therefore gives a continuous
fixed point $p$.

Differentiation of \eqref{eq:elementary-graph-transform}, with
$p_n:=U_n'$, yields
\[
 p_{n+1}(X)
 =H\bigl(p_n(G_n^{-1}(X)),\phi'(G_n^{-1}(X))\bigr).
\]
For a uniformly continuous function $v:\R\to\R$, write
\[
 \omega_v(t):=\sup_{|x-z|\le t}|v(x)-v(z)|;
\]
then $\omega_v(t)\to0$ as $t\to0$.
By \eqref{eq:graph-iterate-rate},
\[
 d_n:=\|G_n^{-1}-G^{-1}\|_\infty
 \le\kappa\|U_n-U\|_\infty
 \le\eps\kappa^{n+1}.
\]
Set $e_n:=\|p_n-p\|_\infty$.
The identity $\mathcal Sp=p$ and
\eqref{eq:slope-derivative-bounds} imply
\begin{align}
 e_{n+1}
 &\le q e_n+\eta_n,
 \label{eq:derivative-convergence-recurrence}\\
 \eta_n
 &:=q\,\omega_p(d_n)+m^{-2}\omega_{\phi'}(d_n)
 \longrightarrow0.\nonumber
\end{align}
For $n>N$, iteration gives
\[
 e_n\le q^{\,n-N}e_N+
       \frac{\sup_{j\ge N}\eta_j}{1-q}.
\]
Letting first $n\to\infty$ and then $N\to\infty$, we obtain $e_n\to0$.
Thus the derivatives converge uniformly; only continuity of $\phi'$
has been used.

For every $x,z\in\R$,
\[
 U_n(z)-U_n(x)=\int_x^z p_n(t)\,dt.
\]
Passing to the limit gives
$U(z)-U(x)=\int_x^z p(t)\,dt$.
Since $p$ is continuous, $U'=p$ everywhere and
$\|U\|_{C^1}\le\eps$.
The graphing function $\Psi_\alpha:=c+U$ satisfies the asserted bound.
Uniqueness among all continuous invariant graphs follows from the
area argument above.
\end{proof}

\subsection{A global parameter estimate}

\begin{proof}[Proof of Theorem~\ref{M1}, part (\ref{item:M1-parameters})]
Fix $\alpha,\alpha'\in\R^2$ and put
$D_0:=\|\Psi_\alpha-\Psi_{\alpha'}\|_\infty$.
By \eqref{induced-lift},
\[
 \|g_\alpha-g_{\alpha'}\|_\infty
 \le|\alpha_1-\alpha_1'|+\lambda D_0.
\]
Using the invariance equations and the bound
$\|\Psi_{\alpha'}\|_{\Lip}\le\eps$ from
Proposition~\ref{prop:preparatory-persistence}, we obtain
\begin{align*}
 |\Psi_\alpha(g_\alpha(x))-\Psi_{\alpha'}(g_\alpha(x))|
 &\le|\alpha_2-\alpha_2'|+\lambda D_0
       +\eps|g_\alpha(x)-g_{\alpha'}(x)|\\
 &\le|\alpha_2-\alpha_2'|+\eps|\alpha_1-\alpha_1'|
       +\lambda(1+\eps)D_0.
\end{align*}
Since $g_\alpha$ is onto, taking the supremum gives
\[
 \bigl(1-\lambda(1+\eps)\bigr)D_0
 \le|\alpha_2-\alpha_2'|+\eps|\alpha_1-\alpha_1'|.
\]
Since $\eps<\lambda^{-1/2}-1$, we have
$\lambda(1+\eps)<\sqrt\lambda<1$.
The coefficient on the left is therefore positive, proving
\eqref{parameter-bound}.
\end{proof}

\subsection{Rotation numbers and their continuity}\label{sec:rotation-numbers}

We recall the displacement estimates needed for Theorem~\ref{M1}(\ref{item:M1-rotations}); for
$n\ge1$, set
\[
 m_n(f):=\min_{x\in[0,1]}(f^n(x)-x),\qquad
 M_n(f):=\max_{x\in[0,1]}(f^n(x)-x).
\]
For $x\le z\le x+1$, monotonicity gives
$f^n(x)\le f^n(z)\le f^n(x)+1$, and hence
\[
 -(z-x)\le
 (f^n(z)-z)-(f^n(x)-x)
 \le1-(z-x).
\]
Since the displacement is periodic,
\begin{equation}\label{eq:rotation-displacement-width}
 0\le M_n(f)-m_n(f)\le1.
\end{equation}
In addition,
\[
 m_{n+k}(f)\ge m_n(f)+m_k(f),\qquad
 M_{n+k}(f)\le M_n(f)+M_k(f).
\]
Fix $k\ge1$ and write $n=\ell k+j$, where $0\le j<k$.
With $m_0=M_0:=0$, these inequalities give
\[
 \ell m_k(f)+m_j(f)\le f^n(x)-x
 \le\ell M_k(f)+M_j(f).
\]
Dividing by $n$ and passing to lower and upper limits, we obtain
\[
 \frac{m_k(f)}k
 \le\liminf_{n\to\infty}\frac{f^n(x)-x}{n}
 \le\limsup_{n\to\infty}\frac{f^n(x)-x}{n}
 \le\frac{M_k(f)}k.
\]
The bounding interval is independent of $x$ and, by
\eqref{eq:rotation-displacement-width}, has length at most $1/k$.
Letting $k\to\infty$ proves that the limit $\rho(f)$ exists and is
independent of $x$.
The same bounds, for each fixed $n\ge1$, give
\begin{equation}\label{eq:rotation-uniform-error}
 m_n(f)\le n\rho(f)\le M_n(f),\qquad
 |f^n(x)-x-n\rho(f)|\le1.
\end{equation}

The real rotation number is continuous in the uniform topology on
increasing degree-one lifts: if $f_j\to f$ uniformly, then
$\rho(f_j)\to\rho(f)$; see \cite[Chapter~I, Section~1, Lemma~1.5]{de}.

\begin{proof}[Proof of Theorem~\ref{M1}, part (\ref{item:M1-rotations})]
Part~(\ref{item:M1-parameters}) and \eqref{induced-lift} give
continuous dependence of $g_\alpha$ on $\alpha$ in the uniform
topology. Hence $\alpha\mapsto\rho(g_\alpha)$ is continuous.
Writing
$U_\alpha:=\Psi_\alpha-\alpha_2/(1-\lambda)$, we have
\[
 g_\alpha(x)-x
 =\alpha_1+\frac{\lambda\alpha_2}{1-\lambda}
       +\lambda U_\alpha(x)+\phi(x).
\]
Proposition~\ref{prop:preparatory-persistence} gives
$\|U_\alpha\|_\infty\le\eps$. Summing along an orbit and
dividing by its length yields
\begin{equation}\label{eq:rotation-parameter-bound}
 \left|\rho(g_\alpha)-
       \left(\alpha_1+\frac{\lambda\alpha_2}{1-\lambda}\right)\right|
 \le\lambda\eps+\|\phi\|_\infty.
\end{equation}
The right side is independent of $\alpha$.
Thus, with either parameter fixed, $\rho(g_\alpha)$ tends to
$-\infty$ and $+\infty$ as the other parameter tends to those limits.
Continuity and the intermediate value theorem imply that every real
rotation number is attained.
\end{proof}

\section{Finite-order averaging and the persistence exponent}\label{normal-section}

Throughout this section, $\alpha_2=0$ and the fixed irrational
$\alpha_1$ satisfies the Brjuno condition \eqref{eq:brjuno-condition}.
We abbreviate $F_\lambda^\phi=F_{\lambda,(\alpha_1,0)}^\phi$.

\subsection{Exactness and a uniform difference equation}

In coordinates $(\theta,I)$, put $\vartheta=I\,d\theta$.
The zero-mean condition on $\phi$ gives a periodic primitive $\Phi$.
Writing $Y=\lambda I+\phi(\theta)$, direct differentiation yields
\begin{equation}\label{eq:zero-flux}
 (F_\lambda^\phi)^*\vartheta-\lambda\vartheta
 =\phi(\theta)\,d\theta+Y\,dY
 =d\bigl(\Phi(\theta)+Y^2/2\bigr).
\end{equation}
This is a global exactness identity, with a single-valued periodic
primitive. If $S^*\vartheta-\vartheta=dA$ and $G=S^{-1}F_\lambda^\phi S$,
then, with $\mathcal B=\Phi+Y^2/2$,
\[
 G^*\vartheta-\lambda\vartheta
 =d\bigl(\lambda A+\mathcal B\circ S-A\circ G\bigr).
\]
Thus exact coordinate changes preserve this property.
An angle-independent averaged map of the form
$(\theta,I)\mapsto(\theta+\Omega(I),\lambda I+b)$ can be exact only
if $b=0$, since the period of its one-form
$G^*\vartheta-\lambda\vartheta$ along a horizontal circle equals $b$.
This excludes a constant averaged radial displacement, but does not
force the averaged angular frequency to remain equal to $\alpha_1$.

We first derive the small-divisor consequence of the Brjuno hypothesis.
Since the nonnegative series in \eqref{eq:brjuno-condition} converges,
\begin{equation}\label{eq:subexp-cf}
 \lim_{n\to\infty}\frac{\log q_{n+1}}{q_n}=0.
\end{equation}
By the standard best-approximation inequalities, for all sufficiently
large $n$ and $q_n\le |k|<q_{n+1}$,
\[
 \|k\alpha_1\|_{\R/\Z}
 \ge\|q_n\alpha_1\|_{\R/\Z}
 >\frac1{q_n+q_{n+1}}.
\]

Consequently, \eqref{eq:subexp-cf} implies
\[
 0\le \frac1{|k|}\log\frac1{\|k\alpha_1\|_{\R/\mathbb Z}}
 \le\frac{\log(q_n+q_{n+1})}{q_n}\longrightarrow0.
\]
Thus the Brjuno assumption implies the subexponential estimate
\begin{equation}\label{eq:subexp-arithmetic}
 \lim_{|k|\to\infty}\frac1{|k|}
 \log\frac1{\|k\alpha_1\|_{\R/\mathbb Z}}=0.
\end{equation}
We use this consequence to obtain uniform cohomological bounds;
it is not an additional hypothesis on $\alpha_1$.

For $\sigma>0$, define
\begin{equation}\label{eq:subexp-divisor-sum}
 \mathcal K_{\alpha_1}(\sigma)
 :=\sum_{k\ne0}
 \frac{e^{-2\pi\sigma|k|}}{\|k\alpha_1\|_{\R/\mathbb Z}}.
\end{equation}
Condition \eqref{eq:subexp-arithmetic} implies that for every
$\varepsilon>0$ there is $M_\varepsilon(\alpha_1)<\infty$ such that
\[
 \|k\alpha_1\|_{\R/\mathbb Z}^{-1}
 \le M_\varepsilon(\alpha_1)e^{\varepsilon|k|},
 \qquad k\ne0.
\]
Taking $\varepsilon=\pi\sigma$ proves that
$\mathcal K_{\alpha_1}(\sigma)<\infty$ for every $\sigma>0$.

\begin{Lemma}
\label{lem:subexp-cohomology}
Let $\alpha_1$ satisfy the Brjuno condition \eqref{eq:brjuno-condition}.
For $1/2\le\lambda\le1$, an integer $j\ge0$, and a zero-mean
periodic function $v$ bounded and holomorphic on $\mathbb T_s$,
the equation
\begin{equation}\label{eq:homological-equation}
 \lambda u(\theta)-\lambda^j u(\theta+\alpha_1)=v(\theta),
 \qquad \langle u\rangle=0,
\end{equation}
has a unique zero-mean analytic solution on every smaller strip.
For $0<\sigma<s$, it satisfies
\begin{equation}\label{eq:homological-s-bound}
 \|u\|_{s-\sigma}
 \le \sum_{k\ne0}
 \frac{e^{-2\pi\sigma|k|}}{\|k\alpha_1\|_{\R/\mathbb Z}}\|v\|_s.
\end{equation}
The estimate is uniform in $\lambda\in[1/2,1]$ and $j\ge0$.
If $v$ is real on the real axis, so is $u$.
\end{Lemma}

\begin{proof}
For $a\ge0$ and real $t$,
$|1-ae^{it}|\ge|\sin(t/2)|$.
Thus for $k\ne0$,
\begin{equation}\label{eq:homological-denominator}
 \begin{split}
 |\lambda-\lambda^je^{2\pi i k\alpha_1}|
 &\ge\lambda|\sin(\pi k\alpha_1)|\\
 &\ge2\lambda\|k\alpha_1\|_{\R/\mathbb Z}
 \ge\|k\alpha_1\|_{\R/\mathbb Z}.
 \end{split}
\end{equation}
Expanding both sides of the equation in Fourier series
and comparing coefficients, we obtain
\[
 \widehat u_0=0,\qquad
 \widehat u_k=
 \frac{\widehat v_k}{\lambda-\lambda^je^{2\pi i k\alpha_1}}
 \quad (k\ne0).
\]
Since $|\widehat v_k|\le\|v\|_s e^{-2\pi s|k|}$,
\[
\begin{aligned}
\sum_{k\ne0}|\widehat u_k e^{2\pi ik\theta}|
&\le
\|v\|_s\sum_{k\ne0}
 \frac{e^{-2\pi\sigma|k|}}{\|k\alpha_1\|_{\R/\mathbb Z}}
<\infty.
\end{aligned}
\]
Thus the Fourier series converges absolutely and uniformly on
$\T_{s-\sigma}$ and defines a holomorphic solution satisfying
\eqref{eq:homological-s-bound}. The coefficient formula determines
the zero-mean solution uniquely. For real $v$, it also gives
$\widehat u_{-k}=\overline{\widehat u_k}$, so $u$ is real on $\T$.

\end{proof}

In the averaging procedure, we solve only for the zero-mean part
of each right-hand side and retain its mean in the normal form.
Indeed, the constant-mode equation is
\[
 (\lambda-\lambda^j)\langle u\rangle=\langle v\rangle.
\]
Its coefficient vanishes for $j=1$ and tends to zero as
$\lambda\to1$ for every fixed $j\ne1$. Retaining the means avoids
these obstructions to a uniform inverse.
\subsection{Finite-order averaging}

For related normal-form and translated-curve constructions, see
\cite{Mas18,Ma}. We use the following implicit coordinate maps.
Given a periodic analytic scalar function $W(\theta,K)$, define
$C_W$ by
\begin{equation}\label{eq:generating-map-definition}
 I=K+W_\theta(\theta,K),\qquad
 \Theta=\theta+W_K(\theta,K),\qquad
 C_W(\theta,I)=(\Theta,K).
\end{equation}
Subscripts denote partial derivatives.
If
$\|W_{\theta K}\|<1/2$ and $\|W_\theta\|$ is smaller than half
the reserved action-domain margin, the iteration
\[
 K_{m+1}=I-W_\theta(\theta,K_m),\qquad K_0=I,
\]
stays in the domain and contracts with ratio at most $1/2$.

Its locally uniform limit solves the first equation analytically
on a smaller domain. The identity
\begin{equation}\label{eq:generating-map-exactness}
 K\,d\Theta-I\,d\theta=d(KW_K-W)
\end{equation}
shows that $C_W$ is exact area-preserving. The map $C_W$ is
unchanged by adding a constant to $W$.

Conversely, let $E:(\theta,I)\mapsto(\Theta,K)$ be an exact
area-preserving analytic degree-one lift whose radial component is sufficiently
$C^1$-close to $I$. Solve for $I=i(\theta,K)$. In these variables the
one-form
\[
 (i-K)\,d\theta+(\Theta-\theta)\,dK
\]
is closed: area preservation says $i_K=\Theta_\theta$.
Its integral around the $\theta$-circle is zero, because it equals
$d\{K(\Theta-\theta)\}-(K\,d\Theta-I\,d\theta)$.
It therefore has a periodic primitive $W$, normalized by
$\langle W(\cdot,0)\rangle=0$, which gives
\eqref{eq:generating-map-definition}.
Integrating this closed form on the covering space constructs $W$;
the zero period ensures periodicity. This construction is analytic
in auxiliary parameters whenever the domains and the radial
invertibility estimates are uniform.

\begin{Lemma}\label{lem:finite-averaging}
Fix integers $N\ge3$ and $r\ge0$. There are constants
$C,\delta_*,s',a>0$, depending only on $N,r,s,\alpha_1$, for which
the following holds uniformly for $\lambda\in[1/2,1]$.
Let $\phi$ be zero-mean, bounded and holomorphic on $\T_s$, and
real on $\T$, with $0<\delta=\|\phi\|_s\le\delta_*$, and put
$\eta=\sqrt\delta$.
There is an exact area-preserving real-analytic coordinate map $S$,
with a holomorphic extension to $\T_{s'}\times\{|I|<2a\}$, satisfying
\[
 \|S-\Id\|_{C^{r+1}(\T\times[-a,a])}\le C\delta.
\]
There is also a real-analytic function $\Omega$ of the action
variable such that, for $(\theta,J)\in\T\times[-1,1]$,
\begin{equation}\label{eq:scaled-finite-normal-form}
 D_\eta^{-1}S^{-1}F_\lambda^\phi SD_\eta(\theta,J)
 =\bigl(\theta+\Omega(\eta J),\lambda J\bigr)
   +\mathcal R(\theta,J),
 \qquad D_\eta(\theta,J)=(\theta,\eta J),
\end{equation}
with
\begin{equation}\label{eq:finite-normal-form-bounds}
 \|\Omega-(\alpha_1+\lambda\,\Id)\|_{C^{r+2}([-a,a])}
 \le C\delta,
 \qquad
 \|\mathcal R\|_{C^{r+1}(\T\times[-1,1])}\le C\eta^N.
\end{equation}
\end{Lemma}

\begin{proof}
Set $f=\phi/\delta$, so $\|f\|_s=1$ and $\langle f\rangle=0$.
Let $P$ be the unique zero-mean periodic analytic function satisfying
\[
 P_\theta=f,\qquad \langle P\rangle=0.
\]
Its Fourier coefficients are
$\widehat P_k=\widehat f_k/(2\pi i k)$ for $k\ne0$ and
$\widehat P_0=0$, giving a uniform bound on every smaller strip.
Keeping $f$ fixed, consider $F_\lambda^{tf}$ for complex $t$ near
zero; the original map corresponds to $t=\delta$.
Set
\[
 \mathcal L_\lambda(\theta,I)=(\theta+\alpha_1,\lambda I),\qquad
 W^{(0)}(\theta,I,t)=\frac\lambda2 I^2-\frac t\lambda P(\theta).
\]
The implicit equations \eqref{eq:generating-map-definition} give
$K=I+tf(\theta)/\lambda$ and $\Theta=\theta+\lambda K$; hence
\begin{equation}\label{eq:exact-kick-factorization}
 F^{tf}_\lambda=\mathcal L_\lambda C_{W^{(0)}}.
\end{equation}
Assign weight $2$ to $t$ and weight $1$ to $I$. An analytic scalar
function is $O_w(m)$ if its Taylor series in $(I,t)$ has no monomial
$t^pI^j$ with $2p+j<m$. Products add weights,
$\partial_\theta$ preserves weight, and $\partial_I$ decreases
weight by at most one. Only finitely many Taylor coefficients will
be used.

The following composition formula determines the induction.
If $C_bC_a=C_c$, write $(\theta_1,I_1)=C_a(\theta,I)$ and
$(\theta_2,K)=C_b(\theta_1,I_1)$. By \eqref{eq:generating-map-definition}, the defining equations are
\[
 \theta_1=\theta+a_I(\theta,I_1),\qquad
 I_1=K+b_\theta(\theta_1,K).
\]
The composite generating function is
\begin{equation}\label{eq:generating-composition}
 c(\theta,K)=a(\theta,I_1)+b(\theta_1,K)
             -(\theta_1-\theta)(I_1-K).
\end{equation}
Indeed, differentiation of the right side cancels the terms involving
$dI_1$ and $d\theta_1$, leaving
$(I-K)d\theta+(\theta_2-\theta)dK$.
If $a=O_w(m)$ and $b=O_w(n)$, with $m,n\geq2$, Taylor expansion
in these equations gives
\begin{equation}\label{eq:weighted-composition-rule}
 c=a+b+O_w(m+n-1).
\end{equation}
To verify this estimate, note that $\theta_1-\theta$ has weight at least $m-1$ and
$I_1-K$ has weight at least $n$. Thus every term in
$a(\theta,I_1)-a(\theta,K)$ has weight at least $(m-1)+n$,
every term in $b(\theta_1,K)-b(\theta,K)$ has weight at least
$n+(m-1)$, and the product of the two increments in
\eqref{eq:generating-composition} has the same lower bound. Additional Taylor factors increase the
weight. For the inverse map, the defining equations first give
an angular displacement of weight at least $m-1$ and a radial
displacement of weight at least $m$. Its normalized generating
function $b$ therefore has weight at least $m$. Applying
\eqref{eq:weighted-composition-rule} to $C_bC_a=\Id$
gives $b=-a+O_w(2m-1)$.

We proceed by induction on the weight $n=2,\ldots,N$.
At the beginning of the $n$th step, write the current map as
$\mathcal L_\lambda C_W$, where $W$ is normalized by
$\langle W(\cdot,0,t)\rangle=0$.
Assume that all terms of weight below $n$ are independent of
$\theta$ and that
\[
W(\theta,I,0)=\frac{\lambda}{2}I^2.
\]
These properties hold initially for $n=2$, since
$W=W^{(0)}=O_w(2)$.

Write the weight-$n$ part of $W$ as
\[
W_n(\theta,I,t)
=
\sum_{2p+j=n}t^p I^j h_{pj}(\theta).
\]
The identity at $t=0$ implies that every coefficient with
$p=0$ is independent of $\theta$.
Thus, only terms with $p\geq1$ may have a nonconstant angular
part.

To determine the next coordinate change, consider
\[
\chi_n(\theta,I,t)
=
\sum_{\substack{2p+j=n\\p\geq1}}
t^p I^j u_{pj}(\theta),
\qquad
\langle u_{pj}\rangle=0,
\]
where the functions $u_{pj}$ are to be chosen.
The defining equations of the generating maps give
\[
\mathcal L_\lambda^{-1}C_{\chi_n}\mathcal L_\lambda
=
C_{\widetilde\chi_n},
\qquad
\widetilde\chi_n(\theta,I,t)
=
\lambda^{-1}
\chi_n(\theta+\alpha_1,\lambda I,t).
\]
Consequently,
\[
C_{\chi_n}^{-1}
(\mathcal L_\lambda C_W)C_{\chi_n}
=
\mathcal L_\lambda
C_{\widetilde\chi_n}^{-1}C_WC_{\chi_n}.
\]

Since $W=O_w(2)$ and
$\chi_n,\widetilde\chi_n=O_w(n)$, the composition and inverse
formulas yield
\begin{equation}\label{eq:weighted-conjugacy-rule}
W^{\rm new}
=
W+\chi_n-\widetilde\chi_n+O_w(n+1).
\end{equation}
Indeed, the composition errors have weight at least
$n+1$, and the inverse error has weight at least
$2n-1\geq n+1$.

The term $t^p I^j u_{pj}(\theta)$ in $\chi_n$ contributes $
\lambda^{-1}t^p(\lambda I)^j
u_{pj}(\theta+\alpha_1)
=
t^p I^j\lambda^{j-1}u_{pj}(\theta+\alpha_1)
$ to $\widetilde\chi_n$.
We require the corresponding weight-$n$ coefficient of
$W^{\rm new}$ to equal $\langle h_{pj}\rangle$.
We therefore choose $u_{pj}$ to solve
\begin{equation}\label{eq:normalizing-coefficients}
\lambda u_{pj}(\theta)
-\lambda^j u_{pj}(\theta+\alpha_1)
=
-\lambda\bigl(h_{pj}(\theta)-\langle h_{pj}\rangle\bigr).
\end{equation}

The right-hand side has zero mean, so
Lemma~\ref{lem:subexp-cohomology} gives a unique zero-mean
analytic solution.
If $h_{pj}$ is bounded on $\mathbb T_\rho$, then
\[
\|u_{pj}\|_{\rho-\sigma}
\leq
2\sum_{k\ne0}
 \frac{e^{-2\pi\sigma|k|}}{\|k\alpha_1\|_{\R/\mathbb Z}}\|h_{pj}\|_\rho,
\qquad
0<\sigma<\rho.
\]
Thus the nonconstant angular part of weight $n$ is removed,
while all lower weights remain unchanged.

Every term of $\chi_n$ contains a factor $t$.
Therefore, for sufficiently small $t$, the implicit
construction above defines an analytic coordinate change,
and
\[
C_{\chi_n}\big|_{t=0}=\mathrm{Id}.
\]
The conjugation consequently leaves the map at $t=0$
unchanged. With the same normalization of its generating
function, we obtain
\[
W^{\rm new}(\theta,I,0)=\frac{\lambda}{2}I^2.
\]
This completes the induction step.
Uniform domains and bounds for the finite sequence of
coordinate changes are established below.

At every step, the cohomological equation uses the fixed
external frequency $\alpha_1$.
The angular averages are retained in the normal form;
no arithmetic condition is imposed on the corrected
frequency $\Omega(0)$.

We verify the uniform analyticity of the construction.
Choose finitely many nested angle--action domains $D_n$,
reserving positive margins for differentiation, cohomological
inversion, and composition, with a nonempty final domain.
Take the initial action radius small enough to accommodate the
angular displacement $\lambda I$.
By \eqref{eq:homological-s-bound}, the cohomological inverses
associated with these fixed strip losses have uniform bounds.

Assume inductively that the current generating function $W^{[n]}$
and map $G_t^{[n]}=\mathcal L_\lambda C_{W^{[n]}}$ are analytic on
$D_n\times\{|t|<\tau_n\}$, that $W^{[n]}$ and the periodic displacement
$G^{[n]}-F_\lambda^0$ have uniform bounds, and that
\[
 W^{[n]}(\theta,I,0)=\frac{\lambda}{2}I^2.
\]
The initial case follows from the explicit expression for $W^{(0)}$. Since $\chi_n$ is divisible by $t$, estimates for its finitely
many Taylor coefficients, \eqref{eq:homological-s-bound},
and Cauchy's estimates give
on a reserved smaller domain and for $|t|<\tau_n/2$,
\[
 \|\partial_\theta\chi_n\|
 +\|\partial_I\chi_n\|
 +\|\partial_{\theta I}\chi_n\|
 \le B_n|t|,
\]
where $B_n\ge1$ is uniform in $\lambda\in[1/2,1]$ and $\|f\|_s\le1$.
Let $d_n>0$ be no larger than any domain margin reserved for this
step, and set
\[
 \tau_n'
 =\min\left\{
 \frac{\tau_n}{2},\frac1{4B_n},\frac{d_n}{4B_n}
 \right\}.
\]
The preceding contraction construction then defines $C_{\chi_n}$
and its inverse analytically, with contraction constants at most
$1/4$ and displacements at most $d_n/4$.
Set
\[
 G_t^{[n+1]}
 =C_{\chi_n}^{-1}G_t^{[n]}C_{\chi_n},
 \qquad
 E_t=\mathcal L_\lambda^{-1}G_t^{[n+1]}.
\]

The periodic displacement of $E_t$ has a uniform analytic bound, and
$E_0(\theta,I)=(\theta+\lambda I,I)$.
Thus, on a still smaller spatial domain, Cauchy's estimates give
\[
\|(E_t)_2-I\|
 +\|\partial_I(E_t)_2-1\|
 \le A_n|t|,
 \qquad |t|<\tau_n'/2,
\]
where $A_n\ge1$ is uniform in $\lambda,f$.
Choosing
 \[
 \tau_{n+1}
 =\min\left\{
 \frac{\tau_n'}2,\frac1{4A_n},\frac{d_n}{4A_n}
 \right\}
\]
allows analytic inversion of the radial equation by the same
contraction argument.

Since $E_t$ is exact area-preserving, the preceding converse
construction provides its normalized generating function
$W^{[n+1]}$.
Integration and Cauchy's estimates give the uniform bounds required
at the next step, while normalization at $t=0$ yields
$W^{[n+1]}(\theta,I,0)=\lambda I^2/2$.
This closes the induction.
All spatial margins are fixed in advance, and there are only
finitely many steps, so $\tau_*=\min_n\tau_n>0$.
Hence the final domains and the required finite-order derivative
bounds can be chosen independently of $\lambda$ and $f$.

Let $S$ be the finite composition of these coordinate maps. Each
$\chi_n$ is divisible by $t$, so its defining equations and a
finite composition give
\[
 \|S-\Id\|_{C^{r+1}}\leq C|t|
\]
on a fixed smaller domain. The final generating function is
\begin{equation}\label{eq:final-generating-remainder}
 W(\theta,I,t)=\overline W(I,t)+Z(\theta,I,t),
 \qquad Z=O_w(N+1),
\end{equation}
where $\overline W$ is the sum of the angular means of weights
at most $N$, and
$\overline W(I,0)=\lambda I^2/2$.
Thus $\partial_I\overline W=\lambda I+O(t)$ with all fixed finite
derivatives on a smaller fixed action disk. The map generated by $\overline W$ leaves the action
coordinate unchanged; its defining equations are $K=I$ and
$\Theta=\theta+\partial_I\overline W(I,t)$.

It remains to estimate the remainder. On fixed smaller domains,
analyticity and
Cauchy's coefficient estimates give constants $M,a_0,b_0>0$ with
\[
 Z=\sum_{2p+j\geq N+1} z_{pj}(\theta)t^pI^j,
 \qquad |z_{pj}(\theta)|\leq M b_0^{-p}a_0^{-j}.
\]
At $t=\eta^2$ and $|I|\leq3\eta$, summing geometric series for
$\eta$ sufficiently small gives
\begin{equation}\label{eq:weighted-analytic-remainder}
 |Z|+|Z_\theta|\leq C\eta^{N+1},\qquad
 |Z_I|\leq C\eta^N.
\end{equation}
The derivative in $\theta$ uses one further fixed strip loss.
For $|I|\le2\eta$, the map
$K\mapsto I-Z_\theta(\theta,K,\eta^2)$ preserves $|K|\le3\eta$
and has derivative $O(\eta^N)<1/2$ for sufficiently small $\eta$.
Its fixed point solves the first generating equation and gives
$K-I=O(\eta^{N+1})$ on this domain.
The second is
\[
 \Theta=\theta+\partial_I\overline W(K,\eta^2)
                  +Z_I(\theta,K,\eta^2).
\]
The boundedness of $\partial_I^2\overline W$ and
\eqref{eq:weighted-analytic-remainder} show that the angular error
relative to $\theta+\partial_I\overline W(I,\eta^2)$ is
$O(\eta^N)$, while the radial error after multiplication by
$\lambda$ is $O(\eta^{N+1})$. Set
$\Omega(I)=\alpha_1+\partial_I\overline W(I,\eta^2)$ and rescale
$I=\eta J$. Both errors are $O(\eta^N)$ on a fixed complex
$J$-disk of radius greater than $1$. Cauchy estimates on a
smaller angle strip and on $|J|\leq1$ give the same bound for
all derivatives through order $r+1$. This proves
\eqref{eq:scaled-finite-normal-form}--\eqref{eq:finite-normal-form-bounds}.
\end{proof}

\subsection{Persistence for every positive exponent}

For a $C^1$ invariant graph $\Gamma$, normal attraction of order $r\ge1$
means that there is a continuous $DF$-invariant splitting
$T_\Gamma(\T\times\R)=T\Gamma\oplus E^s$ and constants $C>0$, $0<q<1$
such that, for every $n\ge0$,
\[
 \sup_{z\in\Gamma}\|DF^n|_{E^s_z}\|\le Cq^n,
 \qquad
 \sup_{z\in\Gamma}\|DF^n|_{E^s_z}\|
 \left(\sup_{z\in\Gamma}\|DF^{-n}|_{T_z\Gamma}\|\right)^r
 \le Cq^n.
\]
We also call this $r$-normal attraction; normal attraction without an
order refers to $r=1$. The following graph-transform argument
constructs the splitting and verifies these inequalities.

\begin{Lemma}\label{lem:elementary-graph-transform}
Fix an integer $r\geq1$, $R>0$, and $B\geq0$. Let $1/2\le\lambda<1$ and put $h=1-\lambda$. Let
\[
 G(x,J)=(A(x,J),C(x,J))
       =\bigl(x+\Omega_0(J),\lambda J\bigr)+Q(x,J)
\]
be $C^{r+1}$ on $\T\times[-R,R]$, with its first component understood
as a degree-one lift. Assume $\det DG=\lambda$ and
$\sup_{|J|\leq R}|\Omega_0'(J)|\leq B$. Put
\[
 K=16(r+2)(1+B),\qquad L=\frac hK,
 \qquad m=1-\frac{h}{8(r+2)}.
\]
There is a positive $c=c(r,B,R)$ such that
\begin{equation}\label{eq:elementary-smallness}
 e:=\|Q\|_{C^{r+1}}\leq ch^2
\end{equation}
implies the existence of a $C^r$ invariant graph $J=u(x)$ satisfying
\begin{equation}\label{eq:elementary-height-slope}
 \|u\|_\infty\leq\frac{\|Q_2\|_\infty}{h},\qquad
 \|u'\|_\infty\leq L.
\end{equation}
The graph is locally attracting and is the only continuous invariant graph
in a sufficiently small neighborhood of itself. Its induced map $g$ satisfies
$g'\geq m$. Moreover, the graph is $r$-normally attracting: it has a
continuous invariant contracting complement to its tangent bundle, with a normal-to-tangential domination ratio at most
$\lambda/m^{r+1}<1$ in a continuous adapted norm.
We give explicit graph-transform estimates in $h$ and use
the classical normal-hyperbolic regularity theorem for the
$C^r$ conclusion; see \cite{HPS,BB13}.
\end{Lemma}

\begin{proof}
All norms of $Q$ are taken on $\T\times[-R,R]$.
Choose
\[
 c\le\min\{R,(4K)^{-1}\},
 \qquad
 a=\frac{\|Q_2\|_\infty}{h}\le\frac R2,
\]
and consider the complete metric space
\[
 \mathcal X=
 \{v\in C^0(\T):\|v\|_\infty\le a,\ \operatorname{Lip}(v)\le L\},
\]
equipped with the uniform distance.

For $|p|\le L$, the choices of the constants give
\[
 d(x,J,p):=A_x+A_Jp
 \ge1-e-(B+e)L\ge m,
 \qquad
 \frac{e+(\lambda+e)L}{1-e-(B+e)L}\le L.
\]
The second inequality follows from
\[
 e(1+L)^2+BL^2\le hL,
\]
using $e\le hL/4$, $L<1/2$, and $BL\le h/48$.
The lower bound for $d$ also holds for the secant slopes of
$f_v(x):=A(x,v(x))$.
Thus $f_v$ is an increasing degree-one homeomorphism, with
$\operatorname{Lip}(f_v^{-1})\le m^{-1}$, and the graph transform
\[
 (\mathcal Tv)(X)
 =
 C\bigl(f_v^{-1}(X),v(f_v^{-1}(X))\bigr)
\]
is well defined.
The height bound
$\lambda a+\|Q_2\|_\infty=a$ and the preceding slope estimate
give $\mathcal T\mathcal X\subset\mathcal X$.
The derivative calculation for $C^1$ graphs extends to Lipschitz
graphs by periodic convolution and uniform convergence of the
inverse maps.

For smooth $v,w\in\mathcal X$, set $v_t=(1-t)v+tw$ and
$x_t=f_{v_t}^{-1}(X)$.
Differentiating at fixed $X$ and using $\det DG=\lambda$ gives
\[
 \frac{d}{dt}(\mathcal Tv_t)(X)
 =
 \frac{\lambda}
 {d(x_t,v_t(x_t),v_t'(x_t))}
 (w-v)(x_t).
\]
Integration and periodic approximation therefore yield
\begin{equation}\label{eq:elementary-C0-contraction}
 \|\mathcal Tv-\mathcal Tw\|_\infty
 \le\frac{\lambda}{m}\|v-w\|_\infty.
\end{equation}
Since $\lambda/m<1$, the iterates
$u_0=0$, $u_{n+1}=\mathcal Tu_n$ converge uniformly to the unique
fixed point $u\in\mathcal X$.
This proves invariance and the asserted height and Lipschitz bounds.

We first establish $u\in C^1$.
Each $u_n$ belongs to $C^{r+1}$.
For $v\in C^2\cap\mathcal X$ and $V=\mathcal Tv$,
twice differentiating $V(f_v(x))=C(x,v(x))$ gives
\[
 V''(f_v(x))
 =
 \frac{\lambda}{d(x,v(x),v'(x))^3}v''(x)
 +P_2(x,v(x),v'(x)).
\]
Here $P_2$ involves only derivatives of $G$ of order at most two
and is bounded on $|J|\le a$, $|p|\le L$.
Moreover,
\begin{equation}\label{eq:elementary-bunching}
 m^{r+2}\ge1-\frac h8>\lambda,
\end{equation}
so $\lambda/m^3<1$.
Consequently, for a finite constant $E_2$ depending on the fixed map,
\[
 \|u_{n+1}''\|_\infty
 \le\frac{\lambda}{m^3}\|u_n''\|_\infty+E_2.
\]
Thus $\sup_n\|u_n''\|_\infty<\infty$.
Arzel\`a--Ascoli applied to $u_n'$, together with the uniform
convergence of $u_n$, gives $u\in C^1$.
In particular, $g:=f_u$ satisfies $g'\ge m$.

In the coordinates $J=u(x)+z$, the differential along $z=0$ is
\[
 \begin{pmatrix}a_1(x)&b(x)\\0&d_1(x)\end{pmatrix},
 \qquad
 a_1=g'\ge m,\quad
 b=A_J(x,u(x)),\quad
 d_1=\frac{\lambda}{a_1}.
\]
These identities follow from invariance and $\det DG=\lambda$.
The affine operator
\[
 (\mathcal Ss)(x)
 =
 \frac{d_1(x)}{a_1(x)}s(g(x))
 -\frac{b(x)}{a_1(x)}
\]
is a contraction on $C^0(\T)$ with constant at most
$\lambda/m^2<1$.
Its fixed point defines an invariant complementary line
spanned by $(s(x),1)$.
In the corresponding adapted norm, normal contraction is bounded
by $\lambda/m$ and inverse tangential expansion by $m^{-1}$.
Hence the graph is $r$-normally attracting, with domination ratio
at most
\[
 \frac{\lambda}{m^{r+1}}<1.
\]
Equivalence of norms on the compact graph transfers these estimates
to any fixed norm, up to a multiplicative constant.

Applying the regularity theorem for normally hyperbolic invariant
submanifolds to $G$ and its $C^1$ invariant graph now gives
$u\in C^r$; see \cite{HPS} and the discussion following
Corollary~2.2 in \cite{BB13}.
Since $\det DG=\lambda>0$, the map $G$ is locally invertible.
Its restriction to the compact invariant graph is a diffeomorphism,
so $G$ is a diffeomorphism onto its image on a sufficiently small
neighborhood of the graph.

Finally, the transverse component in the flattened coordinates is
\[
 Z(x,z)
 =
 C(x,u(x)+z)-u(A(x,u(x)+z)),
 \qquad
 Z(x,0)=0,\quad
 \partial_zZ(x,0)=d_1(x).
\]
Since $d_1\le\lambda/m<1$, there are $\rho>0$ and $\kappa<1$
such that $|Z(x,z)|\le\kappa|z|$ for $|z|\le\rho$,
with this tube contained in the domain.
Thus every orbit in the tube satisfies
$|z_n|\le\kappa^n|z_0|$.
Any other continuous invariant graph in the tube has a surjective
degree-one induced map.
Every point on it therefore has an $n$-step preimage in the tube,
giving $|z|\le\kappa^n\rho$ for every $n$.
Hence it coincides with $u$, proving local attraction and uniqueness.
\end{proof}

\begin{proof}[Proof of Theorem~\ref{M222}]
Fix a Brjuno number $\alpha_1$, $s>0$, $\gamma>0$, and an integer
$r\ge1$. The Brjuno condition implies
\eqref{eq:subexp-arithmetic}, as shown above. Choose an integer
$N>4+4/\gamma$. All constants in the finite averaging lemma are now
allowed to depend on $N,r,s,\alpha_1$, but not on $\lambda$ or on the
particular perturbation. Write $h=1-\lambda$ and
$\delta=\|\phi^\lambda\|_s\le h^\gamma$.

If $\delta=0$, the graph $\Psi_\lambda=0$ has the stated properties;
uniqueness follows from the area argument in Section~\ref{sec:graph-uniqueness}.
Suppose $\delta>0$ and put $\eta=\sqrt\delta$.
Choose $h_0>0$ sufficiently small that $h_0\le1/2$ and
$h_0^\gamma\le\delta_*$, including the domain smallness conditions in
Lemma~\ref{lem:finite-averaging}. The scaled map is
\[
 G=D_\eta^{-1}S^{-1}F_\lambda^{\phi^\lambda}SD_\eta
   =\bigl(\theta+\Omega(\eta J),\lambda J\bigr)+\mathcal R(\theta,J),
 \qquad
 e:=\|\mathcal R\|_{C^{r+1}}\le C_N\delta^{N/2}.
\]
The constants are uniform for $\lambda\in[1/2,1]$ by the
cohomological estimate \eqref{eq:homological-s-bound} and the
finite-domain construction above. In particular,
\[
 \frac{e}{h^2}
 \le C_N h^{\gamma N/2-2}\longrightarrow0
 \quad\hbox{as }h\to0^+.
\]
After decreasing $\delta_*$ if necessary, we have $\|\Omega'\|_\infty\le2$
and $\eta\le1$. Hence
$\Omega_0(J):=\Omega(\eta J)$ satisfies
$\|\Omega_0'\|_\infty\le\eta\|\Omega'\|_\infty\le2$.
The determinant of $G$ is $\lambda$, because $S$ preserves area and
$D_\eta$ has constant determinant. Decrease $h_0$ once more so that
$C_Nh_0^{\gamma N/2-2}\le c(r,2,1)$, where $c$ is the constant in
Lemma~\ref{lem:elementary-graph-transform}.

That lemma, with $R=1$ and $B=2$, gives a locally attracting $C^r$
invariant graph $J=J_*(\theta)$ satisfying
\[
 \|J_*\|_\infty\le C_N\delta^{N/2}/h,
 \qquad
 \|J_*'\|_\infty\le \frac{h}{48(r+2)}.
\]
The graph lies in the normal-form domain and is $r$-normally
attracting.

Write
$S(\theta,I)=(\theta+u_1(\theta,I),I+u_2(\theta,I))$.
The horizontal projection of $SD_\eta$ restricted to this graph is
\[
 H(\theta)=\theta+u_1(\theta,\eta J_*(\theta)).
\]
Since $\|u_i\|_{C^{r+1}}\le C\delta$,
\[
 H'(\theta)
 =1+\partial_\theta u_1(\theta,\eta J_*)
   +\partial_Iu_1(\theta,\eta J_*)\,\eta J_*'
 =1+O(\delta)>0
\]
after decreasing $h_0$. Moreover, $H(\theta+1)=H(\theta)+1$.
Thus $H$ is a degree-one $C^r$ diffeomorphism, and the transported
curve is a $C^r$ graph $y=\Psi_\lambda(x)$ given by
\[
 \Psi_\lambda(H(\theta))
 =\eta J_*(\theta)+u_2(\theta,\eta J_*(\theta)).
\]
Normal attraction and its invariant splitting are preserved by this
coordinate change. The height estimate is
\[
 \|\Psi_\lambda\|_{C^0}
 \le C\delta+C\frac{\delta^{(N+1)/2}}h
 \le C'\delta,
\]
because
\[
 \frac{\delta^{(N-1)/2}}h
 \le h^{\gamma(N-1)/2-1}\longrightarrow0.
\]

The induced lift is the $C^r$ circle diffeomorphism
\[
 g_\lambda(x)=x+\alpha_1+\lambda\Psi_\lambda(x)+\phi^\lambda(x).
\]
We now prove its quadratic rotation-number estimate. This step uses
the zero-mean hypothesis in addition to the finite-order normal form.

\smallskip\noindent
\emph{A rotation up to a quadratic error in the normalized coordinates.}
Let $\widetilde g$ be the lift induced by $G$ on $J_*$ and set
$\nu=\Omega(0)$. Then
\[
 \widetilde g(\theta)
 =\theta+\Omega(\eta J_*(\theta))
       +\mathcal R_1(\theta,J_*(\theta)),
 \qquad g_\lambda=H\widetilde gH^{-1}.
\]
Using $\|\Omega'\|_\infty\le2$ gives
\begin{equation}\label{eq:normalized-quadratic-remainder}
 \|\widetilde g-\Id-\nu\|_\infty
 \le C\left(\frac{\delta^{(N+1)/2}}h+\delta^{N/2}\right)
 \le C\delta^2.
\end{equation}
Indeed,
\[
 \frac{\delta^{(N+1)/2}}{h\delta^2}
 =\frac{\delta^{(N-3)/2}}h
 \le h^{\gamma(N-3)/2-1},
 \qquad
 \frac{\delta^{N/2}}{\delta^2}=\delta^{N/2-2}.
\]
The chosen $N$ makes both exponents positive. Moreover,
\begin{equation}\label{eq:horizontal-near-identity}
 H=\Id+v,\qquad \|v\|_{C^1}\le C\delta,
\end{equation}
as follows from the displayed formula for $H'$ and
$\|J_*'\|_\infty\le Ch$.

\smallskip\noindent
\emph{An exact identity for the mean displacement.}
Since $g_\lambda$ is a $C^1$ circle diffeomorphism,
\eqref{eq:mean-normalization}, with $\alpha_2=0$ and
$\langle\phi^\lambda\rangle=0$, gives
\begin{equation}\label{eq:exact-mean-displacement}
 \int_0^1\bigl(g_\lambda(x)-x\bigr)\,dx=\alpha_1.
\end{equation}
This is an exact identity; no error term is divided by $1-\lambda$.

\smallskip\noindent
\emph{Quadratic change of mean under a near-identity conjugacy.}
Define $g_0=HR_\nu H^{-1}$, where $R_\nu(\theta)=\theta+\nu$.
Equations \eqref{eq:normalized-quadratic-remainder} and
\eqref{eq:horizontal-near-identity} imply
\begin{equation}\label{eq:comparison-with-conjugated-rotation}
 \|g_\lambda-g_0\|_\infty
 \le \operatorname{Lip}(H)
       \|\widetilde g-R_\nu\|_\infty
 \le C\delta^2.
\end{equation}
Substituting $x=H(\theta)$ and using periodicity yields
\[
 \begin{split}
 \int_0^1(g_0(x)-x)\,dx
 &=\int_0^1
   \bigl[\nu+v(\theta+\nu)-v(\theta)\bigr]
   \bigl[1+v'(\theta)\bigr]\,d\theta\\
 &=\nu+\int_0^1
   \bigl[v(\theta+\nu)-v(\theta)\bigr]v'(\theta)\,d\theta.
 \end{split}
\]
The last integral is bounded by
$2\|v\|_\infty\|v'\|_\infty\le C\delta^2$. Combining this
calculation with \eqref{eq:exact-mean-displacement} and
\eqref{eq:comparison-with-conjugated-rotation} gives
\begin{equation}\label{eq:normal-frequency-quadratic}
 |\nu-\alpha_1|\le C\delta^2.
\end{equation}

By \eqref{eq:rotation-uniform-error} with $n=1$ and invariance
of the rotation number under degree-one conjugacy, we obtain
\[
 |\rho(g_\lambda)-\nu|
 =|\rho(\widetilde g)-\nu|
 \le\|\widetilde g-\Id-\nu\|_\infty\le C\delta^2.
\]
Together with \eqref{eq:normal-frequency-quadratic}, this proves
\begin{equation}\label{eq:frequency-quadratic-bound}
 |\rho(g_\lambda)-\alpha_1|
 \le C\delta^2=C\|\phi^\lambda\|_s^2.
\end{equation}
No arithmetic hypothesis on $\nu$ or $\rho(g_\lambda)$ is used,
and no Lipschitz estimate for rotation numbers of two arbitrary
circle maps is invoked.

Uniqueness among all continuous invariant graphs follows from
the area-contraction argument in Section~\ref{S3}.

Taking $\lambda_0=1-h_0$ completes the proof. All choices depend
only on the fixed data $r,\gamma,s,\alpha_1$, and apply to every
allowed perturbation and every $\lambda\in[\lambda_0,1)$.
\end{proof}

\begin{Remark}\label{rem:regularity-quantifiers}
The threshold $\lambda_0$ depends on the prescribed finite
regularity $r$. Theorem~\ref{M222} therefore does not assert
$C^\infty$ or analytic regularity of the invariant graph.
\end{Remark}

\begin{Remark}\label{gnrr}
For $\alpha_2\ne0$, the substitution
$y'=y-\alpha_2/(1-\lambda)$ reduces the map to the same model with
$\alpha_1$ replaced by
$\alpha_1+\lambda\alpha_2/(1-\lambda)$.
For a parameter subfamily, the present proof gives a uniform
persistence threshold if the effective frequencies
$\alpha_{\rm eff}(\lambda)=\alpha_1+\lambda\alpha_2/(1-\lambda)$
satisfy
\[
 \sup_\lambda\mathcal K_{\alpha_{\rm eff}(\lambda)}
                         (\sigma_\ell)<\infty
\]
at each of the finitely many strip losses used in the construction.
A common subexponential small-divisor bound is a sufficient way to
ensure this. Pointwise Brjuno or subexponential assumptions alone do
not give these uniform bounds. For fixed $\alpha_2\ne0$, the effective
frequency passes through rational values as $\lambda\to1$, so this
arithmetic extension must be restricted to an appropriate parameter
subfamily.
\end{Remark}

\section{Rotation numbers on invariant graphs}\label{circle-section}
\label{subsec:regularity-frequency}
The persistence results established above do not prescribe the rotation
number on the invariant graph. In particular, the Brjuno assumption
on $\alpha_1$ in Theorem~\ref{M222} concerns the external parameter,
whereas the perturbed dynamics may have a different rotation number.
We first derive the relevant identities from the graph equation and then
construct analytic perturbations for which the frequency changes.

\begin{Proposition}
\label{prop:actual-frequency}
Let $0<\lambda<1$, let $\phi\in C^1(\T)$ have zero mean, and suppose
that $F^\phi_{\lambda,(\alpha_1,0)}$ admits a $C^1$ invariant graph
$y=\Psi(x)$. The induced lift
\[
 g(x)=x+\alpha_1+\lambda\Psi(x)+\phi(x)
\]
is an increasing $C^1$ degree-one diffeomorphism. Moreover,
\begin{equation}\label{eq:graph-inverse-identity}
 \Psi(x)=x-g^{-1}(x)-\alpha_1,
 \qquad
 \int_0^1\Psi(x)\,dx=0,
 \qquad
 \int_0^1\bigl(g(x)-x\bigr)\,dx=\alpha_1.
\end{equation}
For every $g$-invariant probability measure $\mu$, the real rotation
number associated with this lift satisfies
\begin{equation}\label{eq:frequency-average}
 \rho(g)=\alpha_1+\int_{\mathbb T}\Psi\,d\mu
 =\alpha_1+\frac{1}{1-\lambda}\int_{\mathbb T}\phi\,d\mu.
\end{equation}
Consequently,
\begin{equation}\label{eq:basic-frequency-bound}
 \|\Psi\|_{C^0}\leq\frac{\|\phi\|_{C^0}}{1-\lambda},
 \qquad
 |\rho(g)-\alpha_1|\leq\frac{\|\phi\|_{C^0}}{1-\lambda}.
\end{equation}
\end{Proposition}

\begin{proof}
The invariance equation and the difference of the two components of the
map give
\[
 \Psi(g(x))=\lambda\Psi(x)+\phi(x),
 \qquad
 g(x)=x+\alpha_1+\Psi(g(x)).
\]
The induced map is an orientation-preserving circle homeomorphism, as
observed in Section~\ref{S3}. The second identity gives
$g^{-1}(x)=x-\Psi(x)-\alpha_1$, so both $g$ and $g^{-1}$ are $C^1$.
Differentiating $g^{-1}\circ g=\Id$ shows that $g'$ never vanishes;
since $g$ is increasing, $g'>0$. This proves the regularity assertion
and the first identity in \eqref{eq:graph-inverse-identity}.

Applying \eqref{eq:mean-normalization} with
$\alpha_2=0$ and $\langle\phi\rangle=0$, we obtain
\[
 \langle g-\Id\rangle=\alpha_1,
 \qquad
 \langle\Psi\rangle=0.
\]
This proves the remaining assertions in
\eqref{eq:graph-inverse-identity}.

Next, integrate the invariance equation and
$g(x)-x=\alpha_1+\Psi(g(x))$ against $\mu$. To identify the average
displacement, invariance of $\mu$ and telescoping give
\[
 \int_{\T}\frac{g^n(x)-x}{n}\,d\mu(x)
 =\int_{\T}(g(x)-x)\,d\mu(x).
\]
By the uniform convergence of the displacement averages established in
Section~\ref{S3}, the last integral equals $\rho(g)$. This proves
\eqref{eq:frequency-average}. Finally, surjectivity of $g$ gives
$\|\Psi\|_{C^0}\leq\lambda\|\Psi\|_{C^0}+\|\phi\|_{C^0}$.
This and \eqref{eq:frequency-average} give
\eqref{eq:basic-frequency-bound}.
\end{proof}

The zero-mean condition is imposed with respect to Lebesgue measure,
whereas \eqref{eq:frequency-average} involves an invariant measure of
$g$. These measures need not coincide. The following construction shows
explicitly that $\rho(g)$ can differ from $\alpha_1$ even when the
perturbation is analytic and has zero mean.

\begin{Proposition}
\label{prop:explicit-frequency-shift}
Fix an irrational $\alpha_1$ and a strip width $s_0>0$. For every
sufficiently small nonzero real $t$, there are analytic zero-mean
perturbations $\phi_{\lambda,t}$, uniformly of size $O(|t|)$ on
$\mathbb T_{s_0}$ for $\lambda\in[1/2,1)$, such that
$F^{\phi_{\lambda,t}}_{\lambda,(\alpha_1,0)}$ has an analytic invariant
graph. The dynamics on this graph are analytically conjugate to a rotation of
frequency
\begin{equation}\label{eq:explicit-shift}
 \omega(t)=\alpha_1-\pi t^2\sin(2\pi\alpha_1)+O(t^4),
\end{equation}
which differs from $\alpha_1$ for all sufficiently small $t\ne0$.
\end{Proposition}

\begin{proof}
Let $h_t(x)=x+t\sin(2\pi x)$ and determine $\omega(t)$ from
\begin{equation}\label{eq:fixed-mean-parameter}
 \omega(t)+\pi t^2\sin(2\pi\omega(t))=\alpha_1.
\end{equation}
The derivative with respect to $\omega$ at $(t,\omega)=(0,\alpha_1)$
is one. Hence the implicit function theorem gives a unique solution
near $\alpha_1$, and expansion of \eqref{eq:fixed-mean-parameter} gives
\eqref{eq:explicit-shift}. Define
\[
 g_t=h_t\circ R_{\omega(t)}\circ h_t^{-1},\qquad
 \Psi_t(x)=x-g_t^{-1}(x)-\alpha_1,
\]
and
\[
 \phi_{\lambda,t}(x)
 =g_t(x)+\lambda g_t^{-1}(x)-(1+\lambda)x
   -(1-\lambda)\alpha_1.
\]
The substitution $x=h_t(\theta)$ gives the exact identity
\begin{align*}
 \int_0^1(g_t(x)-x)\,dx
 &=\int_0^1\bigl[\omega(t)+t\{
 \sin(2\pi(\theta+\omega(t)))-\sin(2\pi\theta)\}\bigr]
 (1+2\pi t\cos(2\pi\theta))\,d\theta\\
 &=\omega(t)+\pi t^2\sin(2\pi\omega(t))=\alpha_1.
\end{align*}

By \eqref{eq:mean-normalization}, this identity implies
$\langle\phi_{\lambda,t}\rangle=0$. Direct substitution gives
\[
 F^{\phi_{\lambda,t}}_{\lambda,(\alpha_1,0)}(x,\Psi_t(x))
 =\bigl(g_t(x),\Psi_t(g_t(x))\bigr).
\]
It remains to verify analyticity and smallness on the prescribed strip.
Choose $|t|$ so small that, on $|\operatorname{Im}z|<s_0+1$,
\[
 |t|\sup|\sin(2\pi z)|<\tfrac14,\qquad
 2\pi|t|\sup|\cos(2\pi z)|<\tfrac12.
\]
For each $z\in\T_{s_0+1/2}$, the iteration
$w_{n+1}=z-t\sin(2\pi w_n)$, $w_0=z$, stays in $|w-z|<1/4$
and is a contraction with ratio at most $1/2$. It therefore converges
uniformly on smaller strips to a holomorphic limit, which solves
$h_t(w)=z$ and differs from $z$ by $O_{s_0}(|t|)$.
The same estimates apply to the explicit inverse
$g_t^{-1}=h_t\circ R_{-\omega(t)}\circ h_t^{-1}$.
Consequently, $\|\phi_{\lambda,t}\|_{s_0}\leq C_{s_0}|t|$, with a
constant independent of $\lambda\in[1/2,1)$. Finally,
$\sin(2\pi\alpha_1)\ne0$ because $\alpha_1$ is irrational, proving
the claimed nonzero shift.
\end{proof}

\subsection{Rational dynamics at a fixed irrational parameter}

\begin{proof}[Proof of Theorem~\ref{Mt3}]
Since $\sin(2\pi\alpha_1)\ne0$, choose rational numbers
$\rho_n\to\alpha_1$ such that
\[
 (\alpha_1-\rho_n)\sin(2\pi\alpha_1)>0.
\]
After discarding finitely many terms, $\sin(2\pi\rho_n)$ has the
same sign as $\sin(2\pi\alpha_1)$. Define
\begin{equation}\label{eq:rational-frequency-amplitude}
 t_n:=\left(\frac{\alpha_1-\rho_n}{\pi\sin(2\pi\rho_n)}\right)^{1/2}.
\end{equation}
Then $t_n>0$, $t_n\to0$, and
\[
 \rho_n+\pi t_n^2\sin(2\pi\rho_n)=\alpha_1.
\]
Thus the pair $(t_n,\rho_n)$ satisfies the normalization equation in
Proposition~\ref{prop:explicit-frequency-shift}. Set
\[
 g_n:=h_{t_n}\circ R_{\rho_n}\circ h_{t_n}^{-1},
\]
\[
 \Psi_n(x):=x-g_n^{-1}(x)-\alpha_1,\qquad
 \phi_n(x):=g_n(x)+\lambda g_n^{-1}(x)-(1+\lambda)x-(1-\lambda)\alpha_1.
\]
For sufficiently large $n$, $h_{t_n}'>0$. Hence these formulas define
real-analytic circle maps and periodic functions. The mean-displacement
identity from Proposition~\ref{prop:explicit-frequency-shift} gives
\[
 \int_0^1(g_n(x)-x)\,dx
 =\rho_n+\pi t_n^2\sin(2\pi\rho_n)=\alpha_1.
\]

Equation~\eqref{eq:mean-normalization} gives
$\langle\phi_n\rangle=0$.
The converse statement in Section~\ref{sec:graph-uniqueness}
gives the invariant graph $\Gamma_{\Psi_n}$ with induced lift $g_n$.
Uniqueness among continuous invariant graphs follows from the
area argument in the same section.
Since $g_n=h_{t_n}\circ R_{\rho_n}\circ h_{t_n}^{-1}$,
its rotation number is $\rho_n$.

The fixed-strip inverse construction in
Proposition~\ref{prop:explicit-frequency-shift} applies uniformly to
all sufficiently small $t_n$. Thus $g_n$, $g_n^{-1}$, $\Psi_n$, and
$\phi_n$ are holomorphic on $\T_s$, and
\[
 \|\phi_n\|_s\le C_s t_n\longrightarrow0.
\]
The constant can be chosen uniformly for $0<\lambda<1$, since the
coefficients in the defining formula have absolute value at most two.
Finally, $\phi_n\ne0$: if it vanished, the unique invariant graph of
the unperturbed map would be $y=0$ and its rotation number would be
$\alpha_1$, whereas $\rho_n\in\Q$ and $\alpha_1\notin\Q$.
\end{proof}

For each integer $k\ge0$, the same sequence tends to zero in $C^k$ by
Cauchy's estimate on a disk of radius $s/2$ centered at a real point:
\[
 \|\phi_n\|_{C^k}
 \le \max_{0\le j\le k}j!\left(\frac2s\right)^j\|\phi_n\|_s.
\]

For any fixed $\gamma>0$ and $0<\lambda<1$, taking $n$ sufficiently
large gives $\|\phi_n\|_s<(1-\lambda)^\gamma$.
When $\alpha_1$ is Brjuno and $\lambda$ lies in the range of
Theorem~\ref{M222}, these examples satisfy its hypotheses while
having rational internal rotation numbers.
\subsection{Nontrivial periodic dynamics}

The preceding examples have rational rotation numbers, and every point on
each graph is periodic. A further perturbation produces attracting and
repelling periodic orbits on the graph. Choose one of the rational
frequencies above and write
$\rho_n=p/q$ in lowest terms, with $q\ge3$, and let $t_0>0$ be given by
\eqref{eq:rational-frequency-amplitude}. Put
\[
 f_b(x):=x+p/q+b\sin(2\pi qx),\qquad
 g_{t,b}:=h_t\circ f_b\circ h_t^{-1},\qquad
 M(t,b):=\int_0^1(g_{t,b}(x)-x)\,dx.
\]
After the substitution $x=h_t(\theta)$, the integrand is real analytic
in $(t,b)$ on a fixed neighborhood of $(t_0,0)$. Thus $M$ is real
analytic there, and
\[
 M(t_0,0)=\alpha_1,\qquad
 \partial_tM(t_0,0)=2\pi t_0\sin(2\pi p/q)\ne0.
\]
The implicit function theorem therefore gives $t=t(b)$ with
$t(0)=t_0$ and $M(t(b),b)=\alpha_1$. For
\[
 0<2\pi qb<1-\lambda,
\]
the points $j/q$ form a repelling $q$-cycle of $f_b$, with multiplier
$(1+2\pi qb)^q>1$, and the points $(j+1/2)/q$ form an attracting
$q$-cycle, with multiplier
\[
 \lambda^q<(1-2\pi qb)^q<1.
\]
Indeed, the sine vanishes at these points, translation by $p/q$
permutes each grid in one $q$-cycle, and the derivatives on the two
grids are respectively $1+2\pi qb$ and $1-2\pi qb$.
The rotation number is consequently $p/q$. Substituting $g_{t(b),b}$
in the same formulas for $\Psi$ and $\phi$ gives a zero-mean analytic
perturbation and an analytic invariant graph containing these cycles.
The ambient map has determinant $\lambda$, so along a $q$-cycle its
second multiplier is $\lambda^q$ divided by the tangential multiplier.
The attracting cycle is therefore a periodic sink of the cylinder map, while
the repelling cycle on the graph is a saddle periodic orbit of the cylinder map.

For fixed $q$, the inverse of $f_b$ on a prescribed strip follows
from the contraction iteration
$w\mapsto z-p/q-b\sin(2\pi qw)$ when $b$ is sufficiently small;
this gives analyticity of the construction on the prescribed strip.
First choose the rational approximation so that $t_0$ is sufficiently
small, and then choose $b>0$ sufficiently small. The resulting
$\|\phi\|_s$ is smaller than any prescribed positive bound.
\section{A critical Denjoy construction}\label{S5}

The construction in this section is inspired by Arnaud's
construction of a nondifferentiable invariant curve carrying
Denjoy dynamics~\cite{Arnaud2011}.
We adapt this construction to the dissipative twist-map setting,
with particular attention to the Lipschitz seminorm of the
perturbation and the nondifferentiability points of the invariant graph. We first construct a circle homeomorphism
that is affine on the complementary intervals of a null Cantor set, and
then recover the perturbation and its invariant graph. The estimates are
uniform in the prescribed irrational rotation number $\omega$.

\subsection{Choice of the interval lengths}

Let $\lambda$, $\omega$, $\beta$, and $A$ be as in Theorem~\ref{M3}.
For a positive integer $N$, set
\[
 q_N:=\left(\frac{N+1}{N}\right)^2,\qquad
 p_N:=\left(\frac{N+2}{N+1}\right)^2,\qquad
 u_N:=1-p_N^{-1}.
\]
Choose $N$ such that
\begin{equation}\label{eq:denjoy-large-N}
 (1+\lambda)(q_N-1)<A.
\end{equation}
Since $q_N-1\le3/N$, it suffices to take $N>6/A$. Define
\begin{equation}\label{eq:denjoy-critical-parameter}
 z_N:=\min\left\{\frac{1+\lambda}{1+\lambda-A},
                  1+A+\lambda u_N\right\}.
\end{equation}
The denominator $1+\lambda-A=2\sqrt\lambda$ is positive. Moreover,
$0<u_N<q_N-1<A$ and $z_N>q_N>1$. Indeed, the second entry in the
minimum is larger than $1+A>q_N$, while
\[
 (1+\lambda)(1-q_N^{-1})<A
 \quad\Longrightarrow\quad
 q_N<\frac{1+\lambda}{1+\lambda-A}.
\]
The inequalities $u_N<1$ and
$A+\lambda=1-2\sqrt\lambda+2\lambda<1$ also give $z_N<2$.

For $k\in\Z$, define
\begin{equation}\label{eq:denjoy-gap-lengths}
 e_k:=\begin{cases}
 z_N(N+1)^{-2},&k=0,\\
 (N+|k|)^{-2},&k\ne0,
 \end{cases}
 \qquad Z_N:=\sum_{k\in\Z}e_k,\qquad \ell_k:=Z_N^{-1}e_k.
\end{equation}
Then $\ell_k>0$ and $\sum_k\ell_k=1$. The sequence $(\ell_k)$ increases
strictly up to $k=0$ and decreases strictly thereafter, with limit zero
at both ends. Hence
\begin{equation}\label{eq:denjoy-gap-size}
 \sum_{k\in\Z}|\ell_{k+1}-\ell_k|=2\ell_0,
 \qquad 0<\ell_0\le\frac2N.
\end{equation}
The first identity follows by telescoping. For the second, use
$N^{-2}<e_0<2N^{-2}$ and
\[
 Z_N\ge\sum_{j=0}^{\infty}(N+j)^{-2}
     \ge\int_N^\infty t^{-2}\,dt=N^{-1}.
\]
Set $r_k:=\ell_{k+1}/\ell_k$. The four ratios nearest the central
interval are
\begin{equation}\label{eq:denjoy-modified-ratios}
 r_{-1}=z_N,\qquad r_0=z_N^{-1},\qquad
 r_{-2}=p_N,\qquad r_1=p_N^{-1}.
\end{equation}
Only $r_{-1}$ and $r_0$ depend on the modified value of $e_0$; all other
ratios belong to $[q_N^{-1},q_N]$. In particular,
\begin{equation}\label{eq:denjoy-ratio-properties}
 z_N^{-1}\le r_k\le z_N,\qquad r_k\ne1,\qquad
 \lim_{|k|\to\infty}r_k=1.
\end{equation}

\subsection{The induced circle homeomorphism}

Let $\theta_k:=k\omega\pmod1\in[0,1)$. Since $\omega$ is irrational,
these points are distinct and dense in $\T$.
For $t\in[0,1)$, put
\[
 a(t):=\sum_{\theta_j<t}\ell_j,\qquad
 b(t):=\sum_{\theta_j\le t}\ell_j.
\]
Define $a_k:=a(\theta_k)$, $b_k:=b(\theta_k)$, and
$I_k:=(a_k,b_k)$, with endpoints interpreted on the circle. Then
$b_k-a_k=\ell_k$. If $\theta_j<\theta_k$, density gives
\[
 a_k-b_j=\sum_{\theta_j<\theta_i<\theta_k}\ell_i>0.
\]
Thus the closed intervals $\overline I_k$ are pairwise disjoint and
have the same cyclic order as the points $\theta_k$.
Their complement
\[
 K_N:=\T\setminus\bigcup_{k\in\Z}I_k
\]
is closed and has measure zero, since $\sum_k\ell_k=1$.
It has no isolated point: an isolated point would be a common endpoint
of two complementary intervals. Consequently, $K_N$ is a null Cantor
set. Denote by $E_N$ the set of endpoints of the intervals $I_k$.
Every open arc meeting $K_N$ contains two points of $K_N$ and, between
them, a closed complementary interval. It follows that both the left
endpoints and the right endpoints are dense in $K_N$.

Define the endpoint map by $a_k\mapsto a_{k+1}$ and
$b_k\mapsto b_{k+1}$. It preserves cyclic order. An oriented arc with
endpoints in $E_N$ has length equal to the sum of the lengths of the
complementary intervals it contains, because $K_N$ has measure zero.
For the corresponding image arc, the same sum has $\ell_j$ replaced
by $\ell_{j+1}$. Its length therefore lies between $z_N^{-1}$ and
$z_N$ times the original length. The endpoint map and its inverse
extend continuously to mutually inverse maps on $K_N$.
Extending affinely on each $I_k$, with slope $r_k$, gives an
orientation-preserving circle homeomorphism $g$ such that
$g(I_k)=I_{k+1}$ and $g(K_N)=K_N$.
The same length calculation, allowing partial intervals at the ends
of an arc, gives for every lift and every $x<y$
\begin{equation}\label{eq:denjoy-bilip}
 z_N^{-1}(y-x)\le g(y)-g(x)\le z_N(y-x).
\end{equation}
Thus $g$ and $g^{-1}$ are Lipschitz and absolutely continuous.

The collapsing map is the continuous nondecreasing degree-one map
$P:\R\to\R$ with fibers
\[
 P^{-1}(t)=[a(t),b(t)]\qquad(0\le t<1),
\]
extended by $P(x+1)=P(x)+1$. It is obtained by taking the generalized
inverse of the cumulative distribution of
$\sum_k\ell_k\delta_{\theta_k}$. A jump of this inverse would give
an open interval of base points containing no mass, which is excluded
by the density of the $\theta_k$. This proves continuity. Since
$a(t)=b(t)$ outside the rotation orbit, the only nontrivial fibers are
the closed intervals $\overline I_k$.
By construction, $P\circ g=R_\omega\circ P$ on the circle. Choose the
lift of $g$ for which
\begin{equation}\label{eq:denjoy-semiconjugacy}
 P(g(x))=P(x)+\omega.
\end{equation}
This choice is possible because, for any lift, the difference between
the two sides is a continuous integer-valued function and hence a
constant. As $P-\Id$ is bounded and periodic, iteration yields
\[
 |g^n(x)-x-n\omega|
 \le2\|P-\Id\|_\infty\qquad(n\ge1).
\]
It follows that $\rho(g)=\omega$.
The intervals $g^n(I_0)=I_n$ are pairwise disjoint, so $I_0$ is wandering.
An irrational rotation has no wandering interval, since its iterates
preserve interval length. Hence $g$ is not topologically conjugate to
$R_\omega$.

The left endpoints and the right endpoints form two disjoint $g$-orbits,
each dense in $K_N$. In fact, $K_N$ is a minimal invariant set.
Let $U$ be a nonempty relatively open subset of $K_N$, and choose an
open arc $V$ with $\varnothing\ne V\cap K_N\subset U$.
The perfection and empty interior of $K_N$ provide two closed gaps
contained in $V$. Taking a nonempty open base arc $J$ between their
base points gives
\[
 \varnothing\ne K_N\cap P^{-1}(J)\subset U.
\]
For any $x\in K_N$, the rotation orbit of $P(x)$ meets $J$.
Equation~\eqref{eq:denjoy-semiconjugacy} then implies that the orbit
of $x$ meets $U$, proving minimality.

\subsection{Derivatives on the Cantor set}

Fix $x\in K_N\setminus E_N$ and $\eta>0$. By
\eqref{eq:denjoy-ratio-properties}, there is $M$ such that
$|r_j-1|<\eta$ for $|j|>M$. Since $x$ lies in none of the finitely
many closed intervals $\overline I_j$, $|j|\le M$, a neighborhood of
$x$ avoids all of them. For $y$ in this neighborhood, absolute
continuity and the affine formulas on the gaps give
\[
 |g(y)-g(x)-(y-x)|
 \le\int_{\min\{x,y\}}^{\max\{x,y\}}|g'(t)-1|\,dt
 \le\eta|y-x|,
\]
where the last inequality uses $\Leb(K_N)=0$. Therefore, $g'(x)=1$.
At an endpoint of $I_k$, the same argument applies on the side exterior
to $I_k$: one avoids the finitely many other closed gaps and takes
$y$ on that side. Thus the exterior one-sided derivative is one,
whereas the derivative from within $I_k$ is $r_k$.

On $I_k$, the inverse map is affine with slope $r_{k-1}^{-1}$.
These slopes also tend to one as $|k|\to\infty$. Applying the same
integral estimate to $g^{-1}$ gives $(g^{-1})'=1$ on
$K_N\setminus E_N$. At an endpoint of $I_k$, its exterior derivative
is one and its derivative from within the gap is $r_{k-1}^{-1}$.

\subsection{The perturbation and the invariant graph}

Define
\begin{align}
 \alpha_N&:=\int_0^1(g(x)-x)\,dx,
 \label{eq:denjoy-alpha}\\
 \phi_N(x)&:=g(x)+\lambda g^{-1}(x)-(1+\lambda)x
                  -(1-\lambda)\alpha_N,
 \label{eq:denjoy-phi}\\
 \Psi_N(x)&:=x-g^{-1}(x)-\alpha_N.
 \label{eq:denjoy-graph}
\end{align}
Both functions are periodic and Lipschitz.
Since $g$ is bi-Lipschitz and
$\alpha_N=\langle g-\Id\rangle$,
\eqref{eq:mean-normalization}, applied with
$\alpha=(\alpha_N,0)$, gives $\langle\phi_N\rangle=0$.
The converse statement in Section~\ref{sec:graph-uniqueness} yields
\begin{equation}\label{eq:denjoy-invariance}
 F^{\phi_N}_{\lambda,(\alpha_N,0)}(x,\Psi_N(x))
 =(g(x),\Psi_N(g(x))).
\end{equation}
Thus $\Gamma_{\Psi_N}$ is invariant, with induced map $g$.

The function $\phi_N$ is affine on $I_k$, with slope
\begin{equation}\label{eq:denjoy-phi-slopes}
 d_k=r_k+\frac{\lambda}{r_{k-1}}-(1+\lambda).
\end{equation}
Only $d_{-1}$, $d_0$, and $d_1$ involve the modified ratios $r_{-1}$
and $r_0$. Their values are
\begin{align}
 d_{-1}&=z_N-1-\lambda u_N,\notag\\
 d_0&=-(1+\lambda)(1-z_N^{-1}),
 \label{eq:denjoy-three-slopes}\\
 d_1&=\lambda(z_N-1)-u_N.\notag
\end{align}
By the definition of $z_N$, we have $d_{-1}\le A$ and $d_0\ge-A$,
with equality in at least one of these inequalities. Also,
\[
 -A<-\lambda u_N<d_{-1},\qquad -A<-u_N<d_1,
 \qquad d_{-1}-d_1=(1-\lambda)(z_N-1+u_N)>0,
\]
and $d_0<0$. Therefore, all three absolute values are at most $A$.
For all other $k$, both ratios in \eqref{eq:denjoy-phi-slopes} are
unmodified, whence
\[
 |d_k|\le(1+\lambda)(q_N-1)<A.
\]
Since $K_N$ is null, absolute continuity gives
$|\phi_N(y)-\phi_N(x)|\le A|y-x|$. Equality in the slope bound is
attained on $I_{-1}$ or $I_0$. Consequently,
\begin{equation}\label{eq:denjoy-critical-equality}
 \|\phi_N\|_{\Lip}=\sup_k|d_k|
                         =A=(1-\sqrt\lambda)^2.
\end{equation}

Similarly, $\Psi_N$ is affine on each gap, with
\begin{equation}\label{eq:denjoy-graph-slopes}
 \Psi_N'=1-r_{k-1}^{-1}\quad\hbox{on }I_k.
\end{equation}
This slope is nonzero because $r_{k-1}\ne1$. The preceding derivative
estimates give $\Psi_N'=0$ on $K_N\setminus E_N$. At an endpoint of
$I_k$, its exterior derivative is zero and its derivative from within
the gap is given by \eqref{eq:denjoy-graph-slopes}. Its nondifferentiability set
is therefore exactly $E_N$.
The inverse slopes lie in $[z_N^{-1},z_N]$, and the value $z_N$ is
attained on $I_1$. Hence
\[
 \|\Psi_N\|_{\Lip}=z_N-1.
\]
Both $\Psi_N$ and $\phi_N$ are $C^\infty$ on $\T\setminus K_N$.
Their defining formulas also give $\phi_N'=0$ on
$K_N\setminus E_N$.

\subsection{Uniform and H\"older estimates}

Since $g'=r_k$ on $I_k$ and $K_N$ has measure zero,
\begin{equation}\label{eq:denjoy-variation}
 \int_0^1|g'(x)-1|\,dx
 =\sum_{k\in\Z}\ell_k|r_k-1|
 =\sum_{k\in\Z}|\ell_{k+1}-\ell_k|=2\ell_0.
\end{equation}
Thus $g-\Id$ has oscillation at most $2\ell_0$.
The displacement bounds in Section~\ref{sec:rotation-numbers} place
$\rho(g)=\omega$ between the minimum and maximum of $g-\Id$. Hence
\begin{equation}\label{eq:denjoy-displacement-bound}
 \|g-\Id-\omega\|_\infty\le2\ell_0,
 \qquad |\alpha_N-\omega|\le2\ell_0\le\frac4N.
\end{equation}
The inverse displacement has the same bound, because $y=g(x)$ gives
\[
 g^{-1}(y)-y+\omega=-\bigl(g(x)-x-\omega\bigr).
\]
Subtracting the constant rotation terms in \eqref{eq:denjoy-phi},
we obtain
\begin{equation}\label{eq:denjoy-sup-bound}
 \|\phi_N\|_\infty
 \le\bigl(1+\lambda+(1-\lambda)\bigr)2\ell_0
 =4\ell_0\le\frac8N.
\end{equation}
These bounds depend only on the variation of the interval lengths
and are uniform in $\omega$.

For a periodic Lipschitz function $f$, the bound
\[
 |f(x)-f(y)|\le
 \min\{2\|f\|_\infty,\|f\|_{\Lip}\dist(x,y)\}
\]
implies the interpolation inequality
\[
 [f]_{C^\beta}
 \le (2\|f\|_\infty)^{1-\beta}\|f\|_{\Lip}^{\beta}.
\]
Applying \eqref{eq:denjoy-critical-equality} and
\eqref{eq:denjoy-sup-bound} gives
\[
 \|\phi_N\|_{C^\beta}
 =\|\phi_N\|_\infty+[\phi_N]_{C^\beta}
 \le\frac8N+\left(\frac{16}{N}\right)^{1-\beta} A^\beta
 =O_{\lambda,\beta}(N^{-(1-\beta)}).
\]
This completes the proof of Theorem~\ref{M3}.


\par\medskip\noindent\textbf{Data availability.}
No datasets were generated or analyzed during the current study.
\par\medskip\noindent\textbf{Conflict of interest.}
The authors declare no conflict of interest.
\par\medskip\noindent\textbf{Funding declaration.}
The authors  are partially supported by the National Natural Science Foundation of China (Grant No. 12671217 and 12122109).

\end{document}